\documentclass[11 pt, leqno]{amsart}

\usepackage{amsmath, amsthm, amssymb, amsfonts}
\usepackage{a4wide}

\numberwithin{equation}{section}

\newcommand{\id}{\mathord{\operatorname{id}}}

\newcommand{\N}{\mathbb{N}}

\newcommand{\rE}{\operatorname{E}}
\newcommand{\rV}{\operatorname{V}}

\newcommand{\Aut}{\operatorname{Aut}}

\newcommand{\Gr}{\mathcal{G}}
\newcommand{\GR}{\mathcal{R}}
\newcommand{\autR}{\text{Aut}(\GR)}
\newcommand{\EG}{\rE(\Gr)}
\newcommand{\VG}{\rV(\Gr)}

\newcommand{\aut}{\mathrm{Aut}}

\newcommand{\Ac}{\mathcal A}
\newcommand{\Fc}{\mathcal F}
\newcommand{\Hc}{\mathcal H}
\newcommand{\Rc}{\mathcal R}
\newcommand{\Gc}{\mathcal G}

\newcommand{\F}{\mathbf{F}}

\theoremstyle{plain}
\newtheorem{theorem}{Theorem}[section]
\newtheorem*{theoA}{Theorem A}
\newtheorem*{theoB}{Theorem B}
\newtheorem*{corC}{Corollary C}
\newtheorem*{theoD}{Theorem D}

\newtheorem{lemma}[theorem]{Lemma}
\newtheorem{proposition}[theorem]{Proposition}
\newtheorem{corollary}[theorem]{Corollary}
\theoremstyle{definition}
\newtheorem{definition}[theorem]{Definition}

\newtheorem{remark}[theorem]{Remark}

\begin{document}

\author{Pierre Fima}
\address{Pierre Fima
\newline
Universit\'e de Paris, Sorbonne Universit\'e, CNRS, Institut de Mathématiques de 
Jussieu-Paris Rive Gauche, F-75013, Paris, France.}
\email{pierre.fima@imj-prg.fr}
 \thanks{This research has been partially supported by the ANR Blanc
ANR-14-CE25-0004, acronym GAMME}
\author{Soyoung Moon}
\address{Soyoung Moon
\newline
Universit\'e de Bourgogne, Institut Math\'ematiques de Bourgogne, UMR 5584 CNRS, 
BP 47870, 21078 Dijon cedex 
France}
\email{soyoung.moon@u-bourgogne.fr}
\author{Yves Stalder}
\address{Yves Stalder
\newline
Universit\'e Clermont Auvergne, CNRS, LMBP, F-63000 Clermont-Ferrand, France}
\email{yves.stalder@uca.fr}


\title{Homogeneous Actions on the Random Graph}
\begin{abstract}

\noindent We show that any free product of two (non-trivial) countable groups, one of them being infinite, admits a faithful and homogeneous action on the Random Graph. We also show that a large class of HNN extensions or free products, amalgamated over a finite group, admit such an action and we extend our results to groups acting on trees. Finally, we show the ubiquity of finitely generated free dense subgroups of the automorphism group of the Random Graph whose action on it have all orbits infinite.
\end{abstract}
\maketitle

\section*{Introduction}

\noindent The \textit{Random Graph} (or \textit{Rado graph} or \textit{Erd\H{o}s-R\'enyi graph}) is the unique, up to isomorphism, countable infinite graph $\mathcal{R}$ having the following  property : for  any pair of disjoint finite subsets $(U,V)$ of the set of vertices there exists a vertex adjacent to any vertex in $U$ and not adjacent to any vertex in $V$. Using this property and a model theoretic device called ``back-and-forth" one can show that $\GR$ is \textit{homogeneous}: any graph isomorphism between finite induced subgraphs can be extended to a graph automorphism of $\mathcal{R}$. Hence, the Random Graph plays the same role in graph theory as the Uryshon's space does in metric spaces.

\vspace{0.2cm}

\noindent The Random Graph has been popularized by Erd\H{o}s and R\'enyi in a serie of papers between 1959 and 1968. They showed \cite{ER63} that if a countable graph is chosen at random, by selecting edges independently with probability $\frac{1}{2}$ from the set of $2$-elements subsets of the vertex set, then almost surely the resulting graph is isomorphic to $\mathcal{R}$. Erd\H{o}s and R\'enyi conclude that this Theorem demolishes the theory of infinite random graphs (however, the world of finite random graphs is much less predictable).

\vspace{0.2cm}

\noindent Since almost all countable graphs are isomorphic to $\mathcal{R}$, Erd\H{o}s and Renyi did not give an explicit construction of the Random Graph. However, by using the uniqueness property of $\GR$, it is clear that one may give many different explicit constructions. Such an explicit description was proposed by Rado \cite{Ra64}. The uniqueness property of $\mathcal{R}$ may also be used to show many stability properties  (if small changes are made on $\mathcal{R}$ then the resulting graph is still isomorphic to $\mathcal{R}$) and to construct many automorphisms of $\mathcal{R}$ as well as group actions on $\mathcal{R}$.

\vspace{0.2cm}

\noindent The homogeneity of $\mathcal{R}$ means that its automorphism group $\autR$ is large: it acts transitively on vertices, edges and more generally on finite configurations of any given isomorphism type. We will view it as a closed subset of the Polish group (for the topology of pointwise convergence) of the bijections of the vertices. Hence, it has a natural Polish group topology.

\vspace{0.2cm}

\noindent The goal of this paper is to understand the countable dense subgroups of $\autR$. The first construction of such subgroups was given in \cite{Mac86}, where Macpherson showed that $\autR$ contains a dense free subgroup on 2 generators. More generally, he showed that if $M$ is a $\aleph_0$-categorical structure, then $\Aut(M)$ has a dense free subgroup of rank $\aleph_0$. 
Melles and Shelah \cite{MS94} proved that, if $M$ is a saturated model of a complete theory $T$ with $| M| =\lambda > |T| $, then $\Aut(M)$ has a dense free subgroup of cardinality $2^{\lambda}$. By using the extension property for graphs, Bhattacharjee and Macpherson showed in \cite{BM05} that $\autR$ has a dense locally finite subgroup.

\vspace{0.2cm}

\noindent We call a group action $\Gamma\curvearrowright\GR$ \textit{homogeneous} if, for any graph isomorphism $\varphi\,:\, U\rightarrow V$ between finite induced subgraphs $U,V$ of $\GR$, there exists $g\in\Gamma$ such that $g(u)=\varphi(u)$ for all $u\in U$. The homogeneity of $\mathcal{R}$ means exactly that $\autR\curvearrowright \GR$ is homogeneous. Moreover, it is easy to check that a subgroup $G<\autR$ is dense if and only if the action $G\curvearrowright\GR$ is homogeneous.

\vspace{0.2cm}

\noindent Hence, to understand the countable dense subgroups of $\autR$ one has to identify the class $\mathcal{H}_{\Rc}$ of all countable groups that admit a faithful and homogeneous action on $\GR$. Besides free groups (\cite{Mac86}, \cite{MS94}, \cite{GK03} and \cite{GS15}) and a locally finite subgroup (\cite{BM05}), little is known on groups in $\mathcal{H}_{\Rc}$. There are some obvious obstructions to be in the class $\mathcal{H}_\Rc$: it is easy to deduce from the simplicity of $\autR$, proved in \cite{Tr85}, that any $\Gamma\in\mathcal{H}_\Rc$ must be icc and not solvable (see Corollary \ref{Obstruction}). Our first positive result is the following. We shall use the notion of \textit{highly core-free subgroups}, which is a strengthening of core-freeness, introduced in \cite{FMS15} and recalled in Section $1.2$. As an example, let us note that a finite subgroup in an icc group is highly core-free.

\begin{theoA} If $\Gamma_1,\Gamma_2$ are non-trivial countable groups and $\Gamma_1$ is infinite then $\Gamma_1*\Gamma_2\in\mathcal{H}_{\Rc}$. If $\Sigma<\Gamma_1,\Gamma_2$ is a common finite subgroup such that $\Sigma$ is highly core-free in $\Gamma_1$ and, either $\Gamma_2$ is infinite and $\Sigma$ is highly core-free in $\Gamma_2$, or $\Gamma_2$ is finite and $[\Gamma_2:\Sigma]\geq 2$, then $\Gamma_1\underset{\Sigma}{*}\Gamma_2\in\mathcal{H}_\Rc$.
\end{theoA}

\noindent To prove Theorem $A$, we first show that any infinite countable group $\Gamma$ admits a ``nice" action on $\GR$ (Corollary \ref{CorActR}). To produce this explicit action on $\GR$ we use an inductive limit process. Then starting from an action of $\Gamma:=\Gamma_1*\Gamma_2\curvearrowright\GR$ and an automorphism $\alpha\in\autR$, we construct a natural action $\pi_{\alpha}\,:\, \Gamma\curvearrowright\GR$ and we show that the set $\{\alpha\in\autR\,:\,\pi_{\alpha}\text{ is faithful and homogeneous}\}$ is a dense $G_\delta$ in $\autR$ whenever the initial action is ``nice enough" (Theorem \ref{ThmMain}). We follow the same strategy for amalgamated free products but we need to be more careful since we also have to realize the amalgamated free product relations.

\vspace{0.2cm}

\noindent Using the same strategy, we prove an analogous result for HNN-extensions.

\begin{theoB}
Let $H$ be an infinite countable group, $\Sigma<H$ a finite subgroup and $\theta\,:\,\Sigma\rightarrow H$ an injective group homomorphism. If both $\Sigma$ and $\theta(\Sigma)$ are highly core-free in $H$ then $\text{HNN}(H,\Sigma,\theta)\in\mathcal{H}_\Rc$.
\end{theoB}

\noindent By Bass-Serre theory we obtain the following result.

\begin{corC}
Let $\Gamma$ be a countable group acting, without inversion, on a non-trivial tree $\mathcal{T}$ in the sense of \cite{Se83}. If every vertex stabilizer of $\mathcal{T}$ is infinite and, for every edge $e$ of $\mathcal{T}$ the stabilizer of $e$ is finite and is a highly core-free subgroup of both the stabilizer of the source of $e$ and the stabilizer of the range of $e$, then $\Gamma\in\mathcal{H}_\Rc$.
\end{corC}

\vspace{0.2cm}

\noindent Finally, we study the ubiquity of dense free subgroups of $\aut(\GR)$. Gartside and Knight \cite{GK03} gave necessary and sufficient conditions for a Polish topological group to be ``almost free\footnote{\noindent A Polish group $G$ is \textit{almost free} if for all $n\geq 2$, the set $\{(g_1,\dots,g_n)\in G^n : g_1,\dots,g_n$ freely generates a free subgroup of $G\}$ is a dense $G_\delta$ in $G^n$.}'', and gave applications to permutation groups, profinite groups, Lie groups and unitary groups. In particular, they showed that if $M$ is $\aleph_0$-categorical then  $\Aut(M)$ is almost free. There are abundant results on the ubiquity of free subgroups in various classes of groups. In particular, almost freeness of various oligomorphic\footnote{\noindent A permutation group $G< S(X)$ is called \textit{oligomorphic} if $G$ has only finitely many orbits on $X^n$ for every $n\in \N$. } groups has been shown in \cite{Dr85}, \cite{Ka92}, \cite{GMR93}, \cite{Ca96} and \cite{GK03}. We prove the following result. For $k\geq 2$ and $\bar\alpha=(\alpha_1,\dots,\alpha_k)\in\autR^k$, we denote by $\langle\bar\alpha\rangle$ the subgroup of $\autR$ generated by $\alpha_1,\dots,\alpha_k$, and we set $\mathcal{A}_k:=\{\bar\alpha\in\autR^k\,:\,\langle\bar\alpha\rangle\curvearrowright\GR\text{ has all orbits infinite}\}$. Since $\mathcal{A}_k$ is closed in $\autR^k$, it is a Baire space.

\begin{theoD}For all $k\geq 2$, the set of $\bar\alpha=(\alpha_1,\dots,\alpha_k)\in\autR^k$ such that $\langle\bar\alpha\rangle$ is a free group with basis $\bar\alpha$ and $\langle\bar\alpha\rangle\curvearrowright\mathcal{R}$ is homogeneous with all orbits infinite is a dense $G_\delta$ in $\mathcal{A}_k$.
\end{theoD}

\noindent To prove Theorem $D$ we use the ``back-and-forth" device (Theorem \ref{free-acts-homogeneously}).

\vspace{0.2cm}

\noindent The paper is organized as follows. In Section $1$ we introduce the notations used in the paper about graphs, random extensions and inductive limits. In Section $2$ we introduce the Random Graph and we show how to extend any group action on a finite or countable graph to a group action on the Random Graph. We study the basic properties of the extension and the properties of groups acting homogeneously of $\Rc$. We prove Theorem $A$ in Section $3$, Theorem $B$ in Section $4$, Corollary $C$ in Section $5$ and Theorem $D$ in Section $6$.

\section{preliminaries}

\subsection{Graphs}

A \textit{Graph} $\Gr$ is a pair $(V(\Gr),E(\Gr))$, where $V(\Gr)$ is a set and $E(\Gr)\subset V(\Gr)^2$ is a subset such that $E(\Gr)\cap\Delta=\emptyset$, where $\Delta=\{(u,u)\in V(\Gr)^2\,:\,u\in V(\Gr)\}$, and $(u,v)\in E(\Gr)\Leftrightarrow (v,u)\in E(\Gr)$. For $u,v\in V(\Gr)$, we write $u\sim v$ if $(u,v)\in E(\Gr)$.

\vspace{0.2cm}

\noindent Any subset $U\subset V(\Gr)$ has a natural structure of a graph, called the \textit{induced graph structure on $U$} and denoted by $\Gr_U$, defined by $V(\Gr_U)=U$ and $E(\Gr_U)=E(\Gr)\cap (U\times U)$.

\vspace{0.2cm}

\noindent Let $\Gr_1$, $\Gr_2$ be two graphs. A \textit{morphism of graphs from $\Gr_1$ to $\Gr_2$} is a map $\pi\,:\,V(\Gr_1)\rightarrow V(\Gr_2)$ such that, for all $u,v\in V(\Gr_1)$, $u\sim v\implies \pi(u)\sim\pi(v)$. A morphism $\pi$ is called an \textit{isomorphism} of graphs if $\pi$ is moreover bijective and its inverse $\pi^{-1}\,:\, V(\Gr_2)\rightarrow V(\Gr_1)$ is a morphism of graphs. A morphism of graphs $\pi\,:\,\Gr_1\rightarrow\Gr_2$ is called \textit{open} if
$$\text{for all }u,v\in V(\Gr_1),\,\,u\sim v\Leftrightarrow \pi(u)\sim\pi(v).$$
Note that the inclusion $\iota$ of a subset $U\subset V(\Gr)$ gives an open and injective graph homomorphism $\iota\,:\,\Gr_U\rightarrow\Gr$. Moreover, an injective graph homomorphism $\pi\,:\,\Gr_1\rightarrow\Gr_2$ induces an isomorphism between $\Gr_1$ and the induced graph on $\pi(V(\Gr_1))$ if and only if $\pi$ is open.

\vspace{0.2cm}

\noindent A \textit{partial isomorphism} of a graph $\Gr$ is an isomorphism between induced subgraphs of $\Gr$. For a partial isomorphism $\varphi$, we write $d(\varphi)$ its domain and $r(\varphi)$ its range. We denote by $P(\Gr)$ the set of finite partial isomorphisms i.e. those partial isomorphism for which $d(\varphi)$ is finite (hence $r(\varphi)$ is also finite).

\vspace{0.2cm}

\noindent Given a countable graph $\Gr$, we write Aut$(\Gr)$ the group of isomorphisms from $\Gr$ to $\Gr$. An action of a group $\Gamma$ on the graph $\Gr$ is a group homomorphism $\alpha\,:\,\Gamma\rightarrow\text{Aut}(\Gr)$. We write $\Gamma\curvearrowright\Gr$ for an action of $\Gamma$ on $\Gr$. Let $S(X)$ be the group of bijections of a countable set $X$. It is a Polish group under the topology of pointwise convergence. Note that Aut$(\Gr)\subset S(V(\Gr))$ is a closed subgroup hence, Aut$(\Gr)$ is a Polish group.

\subsection{Group actions}
The main purpose of this section is to introduce the \textit{Random Extension} of a group action on a graph. This notion will be crucial to produce explicit actions on the Random Graph (Section 2.2). Many properties of an action are preserved by its Random Extension (Proposition \ref{PropRandomExt}). However, the freeness property is not preserved; except for actions of torsion-free groups. Nevertheless, a weaker -but still useful- property, that we call \textit{strong faithfulness}, is preserved.

\vspace{0.2cm}

\noindent An action $\Gamma\curvearrowright X$ of a group $\Gamma$ on a set $X$ is called \textit{strongly faithful} if, for all finite subsets $F\subset X$ with $X\setminus F\neq\emptyset$ one has $g(x)=x$ $\forall x\notin F$ $\implies$ $g=1$. When $X$ is finite, an action on $X$ is strongly faithful if and only if it is free. When $X$ is infinite, an action $\Gamma\curvearrowright X$ is strongly faithful if and only if for all $g\in\Gamma\setminus \{ 1\}$, the set $\{x\in X\,:\,gx\neq x\}$ is infinite. Note also that, for an action $\Gamma\curvearrowright X$ with $X$ infinite, one has: almost free (i.e. every non-trivial group element has finitely many fixed points) $\implies$ strongly faithful $\implies$ faithful.

\vspace{0.2cm}

\noindent An action $\Gamma\curvearrowright\Gr$ of a group $\Gamma$ on a graph $\Gr$ is called \textit{homogeneous} if for all $\varphi\in P(\Gr)$ (recall that $d(\varphi)$ and $r(\varphi)$ are supposed to be finite), there exists $g\in\Gamma$ such that $gu=\varphi(u)$ for all $u\in d(\varphi)$. It is easy to see that $\Gamma\curvearrowright\Gr$ is homogeneous if and only if the image of $\Gamma$ in Aut$(\Gr)$ is dense.

\vspace{0.2cm}

\noindent We say that $\Gamma\curvearrowright\Gr$ has \textit{infinite orbits} (resp. is \textit{free}, resp. is \textit{strongly faithful}) if the action $\Gamma\curvearrowright V(\Gr)$ on the set $V(\Gr)$ has infinite orbits i.e. all orbits are infinite (resp. is free, resp. is strongly faithful).

\vspace{0.2cm}

\noindent We say that $\Gamma\curvearrowright\Gr$ \textit{disconnects the finite sets} if for all finite subsets $F\subset V(\Gr)$ there exists $g\in \Gamma$ such that $g F\cap F=\emptyset$ and, for all $u,v\in F$, $gu\nsim v$. Note that if $\Gamma\curvearrowright\Gr$ disconnects the finite sets then it has infinite orbits and the converse holds when $E(\Gr)=\emptyset$.

\vspace{0.2cm}

\noindent We say that $\Gamma\curvearrowright\Gr$ is \textit{non-singular} if for every $u\in V(\Gr)$ and $g\in\Gamma$ we have $gu\nsim u$.
\vspace{0.2cm}

\noindent The notion of a \textit{highly core-free} subgroup $\Sigma<\Gamma$ has been introduced in \cite{FMS15}: it is a strengthening of the notion of core-free subgroup. Recall that, given a nonempty subset $S\subset \Gamma$, the normal core of a subgroup $\Sigma<\Gamma$ relative to $S$ is defined by ${\rm Core}_S(\Sigma)=\cap_{h\in S}h^{-1}\Sigma h$. Then, $\Sigma$ is called \textit{core-free} if ${\rm Core}_\Gamma(\Sigma)=\{1\}$. Now, $\Sigma$ is called \textit{highly core-free} if, for every finite covering of $\Gamma$ with non-empty sets, up to finitely many $\Sigma$-classes, there exists at least one set in the covering for which the associated normal core is trivial. More precisely, it means that, for every finite subset $F\subset \Gamma$, for any $n\geq 1$, for any non-empty subsets $S_1,\dots, S_n\subset\Gamma$ such that $\Gamma-\Sigma F\subset \cup_{k=1}^n S_k$ there exists $1\leq k\leq n$ such that ${\rm Core}_{S_k}(\Sigma)=\{1\}$. Many examples of highly core-free subgroups are given in \cite{FMS15}. Let us mention some of them bellow.
\begin{itemize}
\item For any countable field $K$, any $d\geq 2$, the stabilizer in ${\rm PSL}_d(K)$ of any point in ${\rm P}^1(K^d)$ for the natural action ${\rm PSL}_d(K)\curvearrowright{\rm P}^1(K^d)$ is a highly core-free subgroup of ${\rm PSL}_d(K)$.
\item Any finite subgroup of an icc group is highly core-free (a group is called \textit{icc} if the conjugacy class of every non-trivial group element is infinite).
\item A finite malnormal subgroup of an infinite group is highly core-free ($\Sigma<\Gamma$ is called \textit{malnormal} if $\forall g\in\Gamma-\Sigma$ one has $\Sigma\cap g\Sigma g^{-1}=\{1\}$).
\item If $\Lambda$ and $\Sigma$ are non trivial groups and $\Sigma$ is abelian then $\Sigma$ is highly core-free in $\Gamma=\Lambda*\Sigma$. More generally, an abelian malnormal subgroup of infinite index is highly core-free.
\item For any infinite (commutative) field, $K^*< K^*\ltimes K$ is malnormal, in particular highly core-free by the previous example.
\end{itemize}
Let us finally note that being highly core-free is a strictly stronger property than being core-free. Indeed, let us denote by $S_\infty$ the group of finitely supported bijections of $\N$. Then, the stabilizer of any $n\in\N$ in $S_\infty$ is a core-free but not highly core-free subgroup of $S_\infty$.

\vspace{0.2cm}

\noindent The notion of a highly core-free subgroup $\Sigma<\Gamma$ with respect to an action $\Gamma\curvearrowright X$ on the set $X$ has also been introduced in \cite{FMS15}. It is defined in such a way that $\Sigma<\Gamma$ is highly core-free if and only if it is highly core-free with respect to the left translation action $\Gamma\curvearrowright\Gamma$. We will use here a similar notion for an action $\Gamma\curvearrowright\Gr$ on a \textit{graph} $\Gr$. Let $\Sigma<\Gamma$ be a subgroup and $\Gamma\curvearrowright\Gr$ be an action on the graph $\Gr$. We say that $\Sigma$ is \textit{highly core-free with respect to} $\Gamma\curvearrowright\Gr$ if for every finite subset $F\subset V(\Gr)$ there exists $g\in\Gamma$ such that $gF\cap \Sigma F=\emptyset$, $gu\nsim v$ for all $u,v\in F$, $\Sigma gu\cap\Sigma gv=\emptyset$ for all $u,v\in F$ with $u\neq v$ and $\sigma gu\nsim gv$ for all $u,v\in F$ and all $\sigma\in\Sigma\setminus \{1\}$. In practice, we will use the following equivalent definition (obtained by replacing $F$ by $F_1\cup F_2$): for every finite subsets $F_1,F_2\subset V(\Gr)$ there exists $g\in\Gamma$ such that $gF_1\cap \Sigma F_2=\emptyset$, $gx\nsim u$, $\sigma gx\nsim gx'$ for all $u\in F_2$, $x,x'\in F_1$, $\sigma\in\Sigma\setminus \{1\}$ and $\Sigma gx\cap\Sigma gx'=\emptyset$ for all $x,x'\in  F_1$ with $x\neq x'$.

It is clear that if $\Sigma$ is highly core-free w.r.t. $\Gamma\curvearrowright\Gr$ then the action $\Gamma\curvearrowright\Gr$ disconnects the finite sets. We adopt a terminology slightly different from \cite{FMS15}: if $E(\Gr)=\emptyset$ then $\Sigma$ is highly core-free with respect to the action $\Gamma\curvearrowright\Gr$ on the graph $\Gr$ in the sense explained above if and only if $\Sigma$ is strongly highly core-free with respect to the action $\Gamma\curvearrowright V(\Gr)$ on the set $V(\Gr)$ in the sense of \cite[Definition 1.7]{FMS15}. In particular, if $E(\Gr)=\emptyset$ and $\Gamma\curvearrowright\Gr$ is free then $\Sigma$ is highly core-free with respect to $\Gamma\curvearrowright\Gr$ if and only if $\Sigma$ is highly-core free in $\Gamma$ (see \cite[Definition 1.1 and Lemma 1.6]{FMS15}. Note also that the trivial subgroup is highly core-free w.r.t. $\Gamma\curvearrowright\Gr$ if and only if the action $\Gamma\curvearrowright\Gr$ disconnects the finite sets and, if $\Sigma$ is highly core-free with respect to $\Gamma\curvearrowright\Gr$ then every subgroup of $\Sigma$ is highly core-free with respect to $\Gamma\curvearrowright\Gr$. In particular, if there exists a subgroup $\Sigma$ which is highly core-free with respect to $\Gamma\curvearrowright\Gr$ then the action $\Gamma\curvearrowright\Gr$ disconnects the finite sets. 
\vspace{0.2cm}

\noindent In the sequel we always assume that $\Gr$ is a non-empty graph. \textit{The Random Extension} of $\Gr$ is the graph $\widetilde{\Gr}$ defined by $V(\widetilde{\Gr})=V(\Gr)\sqcup\mathcal{P}_f(V(\Gr))$, where $\mathcal{P}_f(X)$ denotes the set of non-empty finite subsets of a set $X$ and $\sqcup$ denotes the disjoint union, and:
$$E(\widetilde{\Gr})=E(\Gr)\sqcup(\sqcup_{U\in\mathcal{P}_f(V(\Gr))}U\times\{U\})\sqcup(\sqcup_{U\in\mathcal{P}_f(V(\Gr))}\{U\}\times U)$$
i.e., given $a,b\in V(\Gr)$ and $U,V\in\mathcal{P}_f(V(\Gr))$ one has:
$$(a,b)\in E(\widetilde{\Gr})\Leftrightarrow (a,b)\in E(\Gr),\quad (U,V)\notin  E(\widetilde{\Gr}),\quad (a,U)\in E(\widetilde{\Gr})\Leftrightarrow(U,a)\in E(\widetilde{\Gr})\Leftrightarrow a\in U.$$
Note that the inclusion $\iota\,:\,V(\Gr)\rightarrow V(\widetilde{\Gr})$ defines an injective and open graph homomorphism $\iota\,:\,\Gr\rightarrow\widetilde{\Gr}$.

\vspace{0.2cm}

\noindent Given an action $\Gamma\curvearrowright\Gr$ there is a natural action $\Gamma\curvearrowright\widetilde{\Gr}$ for which the map $\iota$ is $\Gamma$-equivariant. Indeed, we take the original action of $\Gamma$ on $V(\Gr)\subset V(\widetilde{\Gr})$ and, for $U\in P_f(V(\Gr))\subset V(\widetilde{\Gr})$ and $g\in\Gamma$, we define the action of $g$ on $U$ by $g\cdot U:=gU:=\{gu\,:\,u\in U\}\in P_f(V(\Gr))$. This defines an action of $\Gamma$ on $\widetilde{\Gr}$ since for all $z\in V(\widetilde{\Gr})$, $U\in P_f(V(\Gr))$ and all $g\in\Gamma$ we have
$$z\sim U\Leftrightarrow z\in U\Leftrightarrow gz\in gU\Leftrightarrow g\cdot z\sim g\cdot U.$$
It is clear that $\iota$ is $\Gamma$-equivariant.

\begin{proposition}\label{PropRandomExt}
 The following holds.
\begin{enumerate}
\item $\Gamma\curvearrowright\widetilde{\Gr}$ is faithful if and only if $\Gamma\curvearrowright\Gr$ is faithful.
\item $\Gamma\curvearrowright\widetilde{\Gr}$ is non-singular if and only if $\Gamma\curvearrowright\Gr$ is non-singular.
\item $\Gamma\curvearrowright\widetilde{\Gr}$ has infinite orbits if and only if $\Gamma\curvearrowright\Gr$ has infinite orbits.
\item $\Gamma\curvearrowright\widetilde{\Gr}$ disconnects the finite sets if and only if $\Gamma\curvearrowright\Gr$ disconnects the finite sets.
\item If $\Gamma$ is torsion-free, $\Gr$ is infinite and $\Gamma\curvearrowright\Gr$ is free then $\Gamma\curvearrowright\widetilde{\Gr}$ is free. If $\Gamma$ has torsion then $\Gamma\curvearrowright\widetilde{\Gr}$ is not free. 
\item If $\Gamma\curvearrowright\Gr$ is strongly faithful and $\Gr$ is infinite then $\Gamma\curvearrowright\widetilde{\Gr}$ is strongly faithful.
\item Let $\Sigma<\Gamma$ be a subgroup. $\Sigma$ is highly core-free w.r.t. $\Gamma\curvearrowright\Gr$ if and only if it is highly core-free w.r.t. $\Gamma\curvearrowright\widetilde{\Gr}$.
\end{enumerate}
\end{proposition}

\begin{proof} 
$(1)$ and $(2)$ are obvious.

\vspace{0.2cm}

\noindent$(3)$. Suppose that  $\Gamma\curvearrowright\Gr$ has infinite orbits. Since $\iota$ is equivariant, it suffices to check that the orbit of any $U\in P_f(V(\Gr))$ is infinite. Suppose that there is $U\in P_f(V(\Gr))$ with a finite orbit. Then there exists $g_1,\dots,g_n\in\Gamma$ such that for all $g\in\Gamma$ there exists $i\in\{1,\dots,n\}$ such that $gU=g_iU$. Hence $\cup_{g\in\Gamma}gU$ is a finite set since $\cup_{g\in\Gamma}gU\subset\cup_{i=1}^n g_iU$. However, since $U\neq\emptyset$, there exists $x\in U$ and $\Gamma x\subset\cup_{g\in\Gamma}gU$ is infinite, a contradiction. The converse is obvious.
\vspace{0.2cm}

\noindent$(4)$. Suppose that $\Gamma\curvearrowright\Gr$ disconnects the finite sets. Let $F\subset V(\widetilde{\Gr})$ be a finite set and write $F=F_1\sqcup F_2$, where $F_1=F\cap V(\Gr)$ and $F_2=F\cap\mathcal{P}_f(V(\Gr))$. Define the finite subset $\widetilde{F}=F_1\cup(\cup_{U\in F_2}U)\subset V(\Gr)$. Since $\Gamma\curvearrowright\Gr$ disconnects the finite sets there exists $g\in\Gamma$ such that $g\widetilde{F}\cap\widetilde{F}=\emptyset$ and, for all $u,v\in\widetilde{F}$, $u\nsim gv$. It follows $gF_1\cap F_1=\emptyset$ and for all $u,v\in F_1$, $gu\nsim v$. Moreover, for all $U,V\in F_2$, $gU\cap V=\emptyset$ hence $gU\neq V$ so $gF_2\cap F_2=\emptyset$. Obvisouly $g F_1\cap F_2=\emptyset$ and, since $gF_1\cap U=\emptyset$ for all $U\in F_2$ we find that $gu\nsim U$ for all $u\in F_1$ and all $U\in F_2$. Hence $gF\cap F=\emptyset$ and for all $u,v\in F$, $gu\nsim v$. This shows that $\Gamma\curvearrowright\widetilde{\Gr}$ disconnects the finite sets. The converse is obvious.
\vspace{0.2cm}

\noindent$(5)$. Suppose $\Gamma\curvearrowright\Gr$ is free and $\Gamma\curvearrowright\widetilde{\Gr}$ is not free. Then, there exists $U\in P_f(V(\Gr))$ and $g\in\Gamma-\{1\}$ such that $gU=U$. Then $g^nU=U$ for all $n\geq 0$. Since $U\neq\emptyset$ there exists $x\in U$ and the set $\{g^nx\,:\,n\geq 0\}$ is finite since it is a subset of $U$. Hence, there exists $n\geq 1$ such that $g^nx=x$. Since $\Gamma\curvearrowright\Gr$ is free we have $g^n=1$ and $\Gamma$ is not torsion free. Suppose now that $\Gamma$ has torsion. Let $g\in\Gamma$, $g\neq 1$, such that $\langle g\rangle$ is a finite subgroup of $\Gamma$. Let $u\in V(\Gr)$. The set $F=\langle g\rangle u\subset V(\Gr)$ is finite and $gF=F$ hence $\Gamma\curvearrowright\widetilde{\Gr}$ is not free.

\vspace{0.2cm}
\noindent$(6)$. Let $F\subset V(\widetilde{\Gr})$ be a finite subset and $g\in\Gamma$ such that $g(x)=x$ for all $x\notin F$. Define $\widetilde{F}=F\cap V(\Gr)$, which is a finite subset of $V(\Gr)$. Since $V(\Gr)-\widetilde{F}\subset V(\widetilde{\Gr})-F$ we have $g(x)=x$ for all $x\in V(\Gr)-\widetilde{F}$. Since  $\Gamma\curvearrowright\Gr$ is strongly faithful it follows that $g=1$.

\vspace{0.2cm}
\noindent$(7)$. Suppose that $\Sigma $ is highly core-free w.r.t. $\Gamma\curvearrowright\Gr$. Let $F\subset V(\widetilde{\Gr})$ be a finite set and write $F=F_1\sqcup F_2$, where $F_1=F\cap V(\Gr)$ and $F_2=F\cap\mathcal{P}_f(V(\Gr))$. Define the finite subset $\widetilde{F}=F_1\cup(\cup_{U\in F_2}U)\subset V(\Gr)$. Since $\Sigma$ is highly core-free w.r.t. $\Gamma\curvearrowright\Gr$ there exists $g\in\Gamma$ such that $g\widetilde{F}\cap\Sigma\widetilde{F}=\emptyset$,  $u\nsim gv$ for all $u,v\in\widetilde{F}$, $\Sigma gu\cap\Sigma gv=\emptyset$ for all $u,v\in\widetilde{F}$ with $u\neq v$ and $\sigma g u\nsim gv$ for all $u,v\in\widetilde{F}$ and all $\sigma\in\Sigma\setminus \{1\}$. Obviously, $gF\cap \Sigma F=\emptyset$ and for all $u,v\in F$, $gu\nsim v$. Let us show that $\Sigma gu\cap\Sigma gv=\emptyset$ for all $u,v\in F$ with $u\neq v$. If both $u,v$ are in $F_1$ or if ($u\in F_1$ and $v\in F_2$) or ($v\in F_2$ and $u\in F_1$) this is obvious. So assume that $U,V\in F_2$. We need to show that $\Sigma gU=\Sigma gV\implies U=V$. We know that for all $u\in U$ and all $v\in V$ we have $u\neq v\implies\Sigma gu\cap\Sigma gv=\emptyset$. Hence, if there exists $x\in\Sigma gU\cap\Sigma gV$ then we find $\sigma,\sigma'\in\Sigma$ such that $\sigma gU=\sigma'gV$. It follows that if $u\in U$ then there exists $v\in V$ such that $\sigma gu=\sigma' gv$ hence, $\Sigma gu=\Sigma gv$ so we must have $u=v\in V$. Hence, $U\subset V$. Similarly, $V\subset U$ so that $U=V$. It remains to show that $\sigma gu\nsim gv$ for all $u,v\in F$ and $\sigma\in\Sigma\setminus \{1\}$. For $\sigma\in\Sigma\setminus \{1\}$ it is clear that $\sigma gu\nsim gv$ for all $u,v\in F_2$ and also for all $u,v\in F_1$. Let $u\in F_1$ and $V\in F_2$. If $\sigma\in\Sigma$ is such that $\sigma gu\sim gV$ then there exists $v\in V\subset \widetilde{F}$ such that $\sigma gu=gv$. Hence $\Sigma gu=\Sigma gv$ which implies that $u=v$. It follows that $\sigma=1$. The converse is obvious.
\end{proof}

\subsection{Inductive limits}

\noindent Let $X_n$ be a sequence of countable sets with injective maps $\iota_n\,:\, X_n\rightarrow X_{n+1}$. Let $l\geq n$. Define the injective map $\iota_{l,n}\,:\, X_n\rightarrow X_l$ by $\iota_{l,n}=\iota_{l-1}\circ\dots\circ\iota_n$ if $l>n$, and $\iota_{n,n}=\id$. Observe that for all $n\leq m\leq l$, $\iota_{l,m}\circ\iota_{m,n}=\iota_{l,n}$. Let $X_\infty$ be the \textit{inductive limit} i.e. $X_\infty=\sqcup X_n/\sim$, where $x\sim y$ if and only if there exists $n,m\in\N$ and $l\geq n,m$ such that $x\in X_n$, $y\in X_m$ and $\iota_{l,n}(x)=\iota_{l,m}(y)$. Observe that we have injections $\iota_{\infty,n}\,:\, X_n\rightarrow X_\infty$, $\iota_{\infty,n}(x)=[x]$, where $[x]$ denotes the class of the element $x\in X_n\subset\sqcup_n X_n$ for the equivalence relation described above. Those injections satisfy $X_\infty=\cup^\uparrow_n\iota_{\infty,n}(X_n)$ and $\iota_{\infty,n+1}\circ\iota_n=\iota_{\infty,n}$ for all $n\in\N$. Actually $X_\infty$ is the unique set, up to a canonical bijection, such that there exists injections $\iota_{\infty,n}\,:\, X_n\rightarrow X_\infty$ satisfying $X_\infty=\cup^\uparrow_n\iota_{\infty,n}(X_n)$ and $\iota_{\infty,n+1}\circ\iota_n=\iota_n$ for all $n\in\N$. Note that $X_\infty$ is at most countable. Actually, it is finite if and only if for all $n\in\N$, $X_n$ is finite and there exists $l\in\N$ such that $\iota_k(X_k)=X_{k+1}$ for all $k\geq l$. 

\vspace{0.2cm}

\noindent Given a sequence of actions $\pi_n\,:\,\Gamma\rightarrow S(X_n)$ of a group $\Gamma$ on the set $X_n$ satisfying $\pi_{n+1}(g)\circ\iota_n=\iota_n\circ\pi_n(g)$ for all $n\in\N$ and $g\in\Gamma$, we define the \textit{inductive limit action} $\pi_\infty\,:\,\Gamma\rightarrow S(X_\infty)$ by $\pi_\infty(g)\circ\iota_{\infty,n}=\iota_{\infty,n}\circ\pi_n(g)$ for all $n\in\N$. The next proposition contains some standard observations on inductive limits. Since this results are well known and very easy to check, we omit the proof.

\begin{proposition}\label{PropLim}
The following holds.
\begin{enumerate}
\item If there exists $n\in\N$ such that $\Gamma\curvearrowright X_n$ is faithful then $\Gamma\curvearrowright X_\infty$ is faithful.
\item $\Gamma\curvearrowright X_\infty$ is free if and only if $\Gamma\curvearrowright X_n$ is free for all $n\in\N$.
\item If $\Gamma\curvearrowright X_n$ is strongly faithful for all $n\in\N$ then $\Gamma\curvearrowright X_\infty$ is strongly faithful.
\item $\Gamma\curvearrowright X_\infty$ has infinite orbits if and only if $\Gamma\curvearrowright X_n$ has infinite orbits for all $n\in\N$.
\end{enumerate}
\end{proposition}

\noindent Let $\Gr_n$ be a sequence of countable graphs with injective graphs homomorphisms $\iota_n\,:\, \Gr_n\rightarrow \Gr_{n+1}$. We define the \textit{inductive limit graph} $\Gr_\infty$ by defining $V(\Gr_\infty)=\cup^\uparrow\iota_{\infty,n}(V(\Gr_n))$ as the inductive limit of the $V(\Gr_n)$ with respect to the maps $\iota_n$ and $E(\Gr_\infty)$ is the set of couples $(u,v)\in V(\Gr_\infty)^2$ for which there exists $k\geq n,m$ with $u=\iota_{\infty,n}(u_0)$, $u_0\in V(\Gr_n)$, $v=\iota_{\infty,m}(v_0)$, $v_0\in V(\Gr_m)$ and $(\iota_{k,n}(u_0),\iota_{k,m}(v_0))\in E(\Gr_k)$. We collect elementary observations on the inductive limit graph in the following proposition.

\begin{proposition}\label{PropGraphLim}
The following holds.
\begin{enumerate}
\item $\iota_{\infty,n}$ is a graph homomorphism for all $n\in\N$.
\item If $\iota_n$ is open for all $n\in\N$ then $\iota_{\infty,n}$ is open for all $n\in\N$.
\item If $\Gamma\curvearrowright\Gr_\infty$ disconnects the finite sets then $\Gamma\curvearrowright\Gr_n$ disconnects the finite sets for all $n\in\N$. The converse holds when $\iota_n$ is open for all $n\in\N$.
\item If $\Gamma\curvearrowright \Gr_\infty$ is non-singular then $\Gamma\curvearrowright \Gr_n$ is non-singular for all $n\in\N$. The converse holds when $\iota_n$ is open for all $n\in\N$.
\item Let $\Sigma<\Gamma$ be a subgroup. If $\Sigma$ is highly core-free w.r.t.  $\Gamma\curvearrowright\Gr_\infty$ then $\Sigma$ is  highly core-free w.r.t. $\Gamma\curvearrowright\Gr_n$ for all $n$. The converse holds when $\iota_n$ is open for all $n$.
\end{enumerate}
\end{proposition}

\begin{proof}
$(1)$ is obvious.
\vspace{0.2cm}

\noindent$(2)$. Suppose that $\iota_n$ is open for all $n\in\N$ and let $u,v\in V(\Gr_n)$ such that $\iota_{\infty,n}(u)\sim\iota_{\infty,n}(v)$. Then, there exists $k\geq 0$ such that $(\iota_{k+n,n}(u),\iota_{k+n,n}(v))\in E(\Gr_{k+n})$ and a proof by induction on $k\geq 0$ shows that $(u,v)\in E(\Gr_n)$.

\vspace{0.2cm}

\noindent$(3)$. Suppose that $\Gamma\curvearrowright\Gr_\infty$ disconnects the finite sets and let $n\in\N$ and $F\subset V(\Gr_n)$ be a finite subset. Let $F'=\iota_{\infty,n}(F)\subset V(\Gr_\infty)$ and take $g\in\Gamma$ such that $gF'\cap F'=\emptyset$ and $gu'\nsim v'$ for all $u',v'\in F'$. It follows that $gF\cap F=\emptyset$ and $\iota_{\infty,n}(gu)\nsim\iota_{\infty,n}(v)$ for all $u,v\in F$. By $(1)$ we have $gu\nsim v$ for all $u,v\in F$. Suppose now that $\iota_n$ is open and $\Gamma\curvearrowright\Gr_n$ disconnects the finite sets for all $n\in\N$. Let $F\subset V(\Gr_\infty)$ be a finite set and take $n\in\N$ large enough so that $F=\iota_{\infty,n}(F')$, where $F'\subset V(\Gr_n)$ is a finite set. Take $g\in\Gamma$ such that $gF'\cap F'=\emptyset$ and $gu'\nsim v'$ for all $u',v'\in F'$. Since $\iota_{\infty,n}$ injective we have $gF\cap F=\emptyset$ and since $\iota_{\infty,n}$ is open (by $(2)$) we have $gu\nsim v$ for all $u,v\in F$.
\vspace{0.2cm}

\noindent$(4)$. Suppose that $\Gamma\curvearrowright \Gr_\infty$ is non-singular and let $n\in\N$, $u\in V(\Gr_n)$ and $g\in\Gamma$. Then $\iota_{\infty,n}(gu)=g\iota_{\infty,n}(u)\nsim\iota_{\infty,n}(u)$ which implies that $gu\nsim u$ since $\iota_{\infty,n}$ is a graph homomorphism. Suppose now that $\iota_n$ is open and $\Gamma\curvearrowright \Gr_n$ is non-singular for all $n\in\N$. Let $u\in V(\Gr_\infty)$ and $g\in\Gamma$. Let $n\in\N$ large enough so that $u=\iota_{\infty,n}(u')$, where $u'\in V(\Gr_n)$. Then $gu'\nsim u'$ implies that $gu=\iota_{\infty,n}(gu')\nsim\iota_{\infty,n}(u')=u$ since $\iota_{\infty,n}$ is open.
\vspace{0.2cm}

\noindent$(5)$.
Suppose that $\Sigma$ is highly core-free with respect to $\Gamma\curvearrowright\Gr_\infty$  and let $n\in\N$ and $F\subset V(\Gr_n)$ be a finite subset. Let $F'=\iota_{\infty,n}(F)\subset V(\Gr_\infty)$ and take $g\in\Gamma$ such that $gF'\cap\Sigma F'=\emptyset$, $gu'\nsim v'$ for all $u',v'\in F'$, $\Sigma gu'\cap\Sigma gv'=\emptyset$ for all $u',v'\in F'$ with $u'\neq v'$ and $\sigma gu'\nsim gv'$ for all $u',v'\in F'$ and all $\sigma\in\Sigma\setminus \{1\}$. It follows that $gF\cap\Sigma F=\emptyset$ and $\Sigma gu\cap\Sigma gv=\emptyset$ for all $u,v\in F$ with $u\neq v$. Moreover, $\iota_{\infty,n}(gu)\nsim\iota_{\infty,n}(v)$ and $\iota_{\infty,n}(\sigma gu)\nsim\iota_{\infty,n}(gv)$ for all $u,v\in F$ and all $\sigma\in\Sigma\setminus \{1\}$. By $(1)$ we have $gu\nsim v$ and $\sigma gu\nsim gv$ for all $u,v\in F$ and all $\sigma\in\Sigma\setminus \{1\}$. Suppose now that $\iota_n$ is open and $\Sigma$ is highly core-free w.r.t. $\Gamma\curvearrowright\Gr_n$ for all $n\in\N$. Let $F\subset V(\Gr_\infty)$ be a finite set and take $n\in\N$ large enough so that $F=\iota_{\infty,n}(F')$, where $F'\subset V(\Gr_n)$ is a finite set. Take $g\in\Gamma$ such that $gF'\cap\Sigma F'=\emptyset$, $gu'\nsim v'$ and $\sigma gu'\nsim gv'$ for all $u',v'\in F'$ and all $\sigma\in\Sigma\setminus \{1\}$ and $\Sigma gu'\cap\Sigma gv'=\emptyset$ for all $u',v'\in F'$ with $u'\neq v'$. Since $\iota_{\infty,n}$ injective we have $gF\cap\Sigma F=\emptyset$ and $\Sigma gu\cap\Sigma gv=\emptyset$ for all $u,v\in F$ with $u\neq v$. Moreover, since $\iota_{\infty,n}$ is open (by $(2)$) we have $gu\nsim v$ and $\sigma gu\nsim gv$ for all $u,v\in F$ and all $\sigma\in\Sigma\setminus \{1\}$.
\end{proof}
\section{The Random Graph}

\subsection{Definition of the Random Graph}Given a graph $\Gr$ and subsets $U,V\subset V(\Gr)$ we define $\Gr_{U,V}$ as the induced subgraph on the subsets of vertices
$$V(\Gr_{U,V}):=\{z\in V(\Gr)\setminus (U\cup V)\,:\, z\sim u\, , \forall u\in U \text{ and }z\nsim v\, ,\forall v\in V\}.$$
Note that $V(\Gr_{U,V})$ may be empty for some subsets $U$ and $V$. To ease the notations we will denote by the same symbol $\Gr_{U,V}$ the induced graph on $V(\Gr_{U,V})$ and the set of vertices $V(\Gr_{U,V})$. 

\begin{definition}
We say that a graph $\Gr$ has \textit{property $(R)$} if, for any disjoint finite subsets $U,V\subset V(\Gr)$, $\Gr_{U,V}\neq\emptyset$.
\end{definition}

\noindent Note that a graph with property $(R)$ is necessarily infinite. We recall the following well-known result (See e.g. \cite{Cam97} or \cite{Cam99}) that will be generalized later (Proposition \ref{PropExtension}).

\begin{proposition}\label{Prop-Extension}
Let $\Gr_1$, $\Gr_2$ be two infinite countable graphs with property $(R)$ and $A\subset V(\Gr_1)$, $B\subset V(\Gr_2)$ be finite subsets. Any isomorphism between the induced graphs $\varphi\,:\,(\Gr_1)_A\rightarrow (\Gr_2)_B$ extends to an isomorphism $\overline{\varphi}\,:\,\Gr_1\rightarrow\Gr_2$.
\end{proposition}

\noindent There are many ways to construct a countable graph with property $(R)$. Proposition \ref{Prop-Extension} shows that a countable graph with property $(R)$ is unique, up to isomorphism. Such a graph is denoted by $\mathcal{R}$ and called \textit{the Random Graph}. Given an infinite countable set $V$ of vertices, Erd\H{o}s and R\'{e}nyi proved \cite{ER63} that putting (independently) an edge between any pair $\{u,v\}$ of vertices with probability $1/2$, the resulting graph will have property ($R$) with probability $1$. This result motivates the name \textit{Random Graph}.

\vspace{0.2cm}

\noindent Proposition \ref{Prop-Extension} also implies that every $\varphi\in P(\mathcal{R})$ admits an extension $\overline{\varphi}\in\autR$ (i.e. $\overline{\varphi}\vert_{d(\varphi)}=\varphi$). Proposition \ref{Prop-Extension} is also useful to show stability properties of the graph $\GR$ as done in the next Proposition.

\begin{proposition}\label{PropStability}The following holds.
\begin{enumerate}
\item For every finite subset $A\subset V(\mathcal{R})$, the induced subgraph $\GR_{V(\mathcal{R})\setminus A}$ on $V(\mathcal{R})\setminus A$ is isomorphic to $\mathcal{R}$.
\item For all finite and disjoint $U,V\subset V(\mathcal{R})$, the graph $\mathcal{R}_{U,V}$ is isomorphic to $\mathcal{R}$.
\item For every non-empty countable graph $\Gr$, the inductive limit $\Gr_\infty$ of the sequence $\Gr_0=\Gr$ and $\Gr_{n+1}=\widetilde{\Gr}_n$ (with $\iota_n\,:\,\Gr_n\rightarrow\Gr_{n+1}$ the inclusion) is isomorphic to $\GR$. In particular, every countable graph is isomorphic to an induced subgraph of $\GR$.
\end{enumerate}
\end{proposition}

\begin{proof}
\noindent $(1).$ Let $U,V\subset V(\GR)\setminus A$ be disjoint finite subsets. Apply property $( R)$ with the disjoint finite subsets $U'=U\sqcup A,V\subset V(\GR)$ and get $z\in V(\GR)\setminus (U'\cup V)=(V(\GR)\setminus A)\setminus (U\cup V)$ such that $z\sim u$ for all $u\in U$ and $z\nsim v$ for all $v\in V$. Hence the countable graph induced on $V(\GR)\setminus A$ has property $( R )$.

\vspace{0.2cm}
\noindent $(2)$. Since $\GR_{U,V}$ is at most countable, it suffices to check that $\GR_{U,V}$ has property $( R)$ and it is left to the reader.

\vspace{0.2cm}
\noindent$(3)$. Since $\Gr$ is either finite and non-empty or infinite countable, the inductive limit is infinite countable and it suffices to check that it has property $( R)$. Since $\iota_n$ is open for all $n\in\N$ it follows from Proposition \ref{PropGraphLim} that we may and will assume that $\iota_{\infty,n}=\id$ and $(\Gr_n)_n$ is an increasing sequence of induced subgraphs of $\Gr_\infty$ such that $\Gr_\infty=\cup^{\uparrow}\Gr_n$. Let $U,V$ be two finite and disjoint subsets of the inductive limit. Let $n$ large enough so that $U,V\varsubsetneq V(\Gr_n)$. If $U\neq\emptyset$, we consider the element $z=U\in P_f(V(\Gr_n))\subset V(\Gr_{n+1})\subset V(\Gr_\infty)$ and if $U=\emptyset$, we consider any element $z=\{x\}\in P_f(V(\Gr_n))\subset V(\Gr_{n+1})\subset V(\Gr_\infty)$ for $x\in V(\Gr_n)\setminus  V$. Then, by definition of the Random extension, we have, for all $u\in U$, $(u, z)\in E(\Gr_{n+1})$ and, for all $v\in V$, $(v, z)\notin E(\Gr_{n+1})$. Since $\Gr_{n+1}$ is the induced subgraph on $V(\Gr_{n+1})$ we have $z\in(\Gr_{\infty})_{U,V}$. The last assertion also follows from Proposition \ref{PropGraphLim} since the inclusion of $\Gr_0=\Gr$ in the inductive limit is open.
\end{proof}

\begin{remark}
The construction of Proposition \ref{PropStability}, assertions $(3)$, shows the existence of $\GR$. This construction may also be performed starting with any countable graph $\Gr$, even $\Gr=\emptyset$, and replacing, in the construction of the Random Extension, $\mathcal{P}_f(X)$ by all the finite subsets of $X$ (even the empty one). The resulting inductive limit is again isomorphic to $\GR$.
\end{remark}

\noindent We shall need the following generalization of Proposition \ref{Prop-Extension}. The proof is done by using the ``back-and-forth" device.

\begin{proposition}\label{PropExtension}
For $k=1,2$, let $\pi_k\,:\,\Sigma\curvearrowright\GR$ be two free and non-singular actions of the finite group $\Sigma$. For any partial isomorphism $\varphi\in P(\GR)$ such that $\pi_1(\Sigma)d(\varphi)=d(\varphi)$ and $\pi_2(\Sigma)r(\varphi)=r(\varphi)$ and $\varphi\pi_1(\sigma)=\pi_2(\sigma)\varphi$ for all $\sigma\in\Sigma$, there exists $\overline{\varphi}\in\autR$ such that $\overline{\varphi}\pi_1(\sigma)=\pi_2(\sigma)\overline{\varphi}$ for all $\sigma\in\Sigma$ and $\overline{\varphi}\vert_{d(\varphi)}=\varphi$.
\end{proposition}

\begin{proof}
Write $V(\GR)\setminus d(\varphi)=\sqcup_{k=1}^\infty\pi_1(\Sigma)x_k$ and $V(\GR)\setminus r(\varphi)=\sqcup_{k=1}^\infty\pi_2(\Sigma)y_k$. We define inductively pairwise distinct integers $k_n$, pairwise distinct integers $l_n$, subsets 
$$A_n=d(\varphi)\sqcup \pi_1(\Sigma)x_{k_1}\sqcup\ldots\sqcup \pi_1(\Sigma)x_{k_n}\subset V(\GR),$$ $$B_n=r(\varphi)\sqcup\pi_2(\Sigma)y_{l_1}\sqcup\dots\sqcup\pi_2(\Sigma) y_{l_n}\subset V(\GR)$$ and isomorphisms $\varphi_n\,:\,A_n\rightarrow B_n$ such that $\varphi_n\vert_{d(\varphi)}=\varphi$, $\varphi_n\pi_1(\sigma)=\pi_2(\sigma)\varphi_n$ for all $\sigma\in\Sigma$, $A_n\subset A_{n+1}$, $B_n\subset B_{n+1}$, $\varphi_{n+1}\vert_{A_n}=\varphi_n$ and $V(\GR)=\cup A_n=\cup B_n$. Once it is done, the Proposition is proved since we can define $\overline{\varphi}\in\autR$, $x\in A_n\mapsto \varphi_n(x)$ which obviously satisfies $\overline{\varphi}\pi_1(\sigma)=\pi_2(\sigma)\overline{\varphi}$ for all $\sigma\in\Sigma$ and $\overline{\varphi}\vert_{d(\varphi)}=\varphi$.

\vspace{0.2cm}

\noindent Define $A_0=d(\varphi)$, $B_0=r(\varphi)$ and $\varphi_0=\varphi$. If step $n$ is constructed and $n$ is even define $k_{n+1}=\text{Min}\{k\geq 1\,:\,x_k\notin A_n\}$, $U=\{x\in A_n\,:\,x\sim x_{k_{n+1}}\}$, $V=A_n\setminus  U$ and
$$W=\{y\in V(\GR)\setminus  B_n\,:\,y\sim \varphi_n(x)\text{ for all }x\in U\text{ and }y\nsim \varphi_n(x)\text{ for all }x\in V\}.$$
By property $(R )$, $W$ is non-empty. Let $l_{n+1}=\text{Min}\{k\geq 1\,:\,\pi_2(\Sigma)y_k\cap W\neq\emptyset\}$. Replacing $y_{l_{n+1}}$ by an element in $\pi_2(\Sigma)y_{l_{n+1}}$ we may and will assume that $y_{l_{n+1}}\in W$. Define $A_{n+1}=A_n\sqcup\pi_1(\Sigma)x_{k_{n+1}}$ and $B_{n+1}=B_n\sqcup\pi_2(\Sigma)y_{l_{n+1}}$. By freeness, we may define a bijection $\varphi_{n+1}\,:\, A_{n+1}\rightarrow B_{n+1}$by $\varphi_{n+1}\vert_{A_n}=\varphi_n$ and $\varphi_{n+1}(\pi_1(\sigma)x_{k_{n+1}})=\pi_2(\sigma)y_{k_{n+1}}$, for all $\sigma\in\Sigma$. By construction, it is an isomorphism between the induced graphs. Indeed if $x\in A_n$ is such that $x\sim\pi_1(\sigma)x_{k_{n+1}}$ for some $\sigma\in\Sigma$ then $\pi_1(\sigma^{-1})x\sim x_{k_{n+1}}$ so $y_{l_{n+1}}\sim\varphi_n(\pi_1(\sigma^{-1})x)=\pi_2(\sigma^{-1})\varphi_n(x)$ hence, $\varphi_{n+1}(x)=\varphi_n(x)\sim\pi_2(\sigma)y_{l_{n+1}}=\varphi_{n+1}(\pi_1(\sigma)x_{k_{n+1}})$. Since $\varphi_n$ is a graph isomorphism and since both actions are non-singular (which means that there is no edges on induced subgraphs of the form $\pi_k(\Sigma)z$) it shows that $\varphi_{n+1}$ is also a graph isomorphism. By construction we also have that $\varphi_{n+1}\pi_1(\sigma)=\pi_2(\sigma)\varphi_{n+1}$ for all $\sigma\in\Sigma$.
\vspace{0.2cm}

\noindent If $n$ is odd define $l_{n+1}=\text{Min}\{k\geq 1\,:\,y_k\notin B_n\}$, $U=\{y\in B_n\,:\,y\sim y_{l_{n+1}}\}$, $V=B_n\setminus  U$ and
$$W=\{x\in V(\GR)\setminus  A_n\,:\,x\sim  \varphi_n^{-1}(y)\text{ for all }y\in U\text{ and }x\nsim \varphi_n^{-1}(y)\text{ for all }y\in V\}.$$
By property $(R)$, $W$ is non-empty. Let $k_{n+1}=\text{Min}\{k\geq 1\,:\,x_k\in W\}$. We may and will assume that $x_{k_{n+1}}\in W$. Define $A_{n+1}=A_n\sqcup\pi_1(\Sigma)x_{k_{n+1}}$, $B_{n+1}=B_n\sqcup\pi_2(\Sigma)y_{l_{n+1}}$. By construction, the map $\varphi_{n+1}\,:\, A_{n+1}\rightarrow B_{n+1}$ defined by $\varphi_{n+1}\vert_{A_n}=\varphi_n$ and $\varphi_{n+1}(\pi_1(\sigma)x_{k_{n+1}})=\pi_2(\sigma)y_{k_{n+1}}$, $\sigma\in\Sigma$, is an isomorphism between the induced graphs.

\vspace{0.2cm}

\noindent Let us show that $\cup A_n=d(\varphi)\sqcup\bigsqcup_{n=1}^\infty\pi_1(\Sigma)x_{k_n}=V(\GR)$. It suffices to show that $\{k_n\,:\,n\geq 1\}=\N^*$. Suppose that there exists $s\in\N$, $s\neq k_n$ for all $n\geq 1$. Since the elements $k_n$ are pairwise distinct, the set $\{k_{2n+1}\,:\,n\geq 0\}$ is not bounded. Hence, there exists $n\in\N$ such that $s<k_{2n+1}$. By definition, we have $k_{2n+1}=\text{Min}\{k\geq 1\,:\,x_k\notin A_{2n}\}$. However we have $x_s\notin A_{2n}$ and $s<k_{2n+1}$, a contradiction. The proof of $\cup B_n=V(\GR)$ is similar.

\end{proof}

\subsection{Induced action on the Random Graph}

Let $\Gamma\curvearrowright\Gr$ be an action of a group $\Gamma$ on a non-empty countable graph $\Gr$ and consider the sequence of graphs $\Gr_0=\Gr$ and $\Gr_{n+1}=\widetilde{\Gr}_n$ with the associated sequence of actions $\Gamma\curvearrowright\Gr_n$. By Proposition \ref{PropStability}, assertion $(3)$, the inductive limit action defines an action of $\Gamma$ on $\GR$. We call it the \textit{induced action of $\Gamma\curvearrowright\Gr$ on $\GR$}. We will show in Corollary \ref{Cor-ActionR} that many properties on the action $\Gamma\curvearrowright\Gr$ are preserved when passing to the induced action $\Gamma\curvearrowright\GR$. However, freeness is not preserved and one has to consider a weaker notion that we call property $(F )$.

\begin{definition}
We say that an action $\Gamma\curvearrowright\Gc$ has \textit{property $(F)$} if for all finite subsets $S\in\Gamma\setminus \{1\}$ and $F\subset V(\Gc)$, there exists $x\in V(\Gr)\setminus F$ such that $x\nsim u$ for all $u\in F$ and $gx\neq x$ for all $g\in S$.
\end{definition}
\noindent Note that any action with property $(F )$ is faithful and any free action on $\mathcal{R}$ has property $(F )$.

\begin{corollary}\label{Cor-ActionR}
Let $\Gamma\curvearrowright\Gr$ be an action of a group $\Gamma$ on a non-empty countable graph $\Gr$ and write $\Gamma\curvearrowright\GR$ the induced action. The following holds.
\begin{enumerate}
\item If $\Gamma\curvearrowright\Gr$ is faithful then $\Gamma\curvearrowright\GR$ is faithful.
\item $\Gamma\curvearrowright\Gr$ has infinite orbits if and only if $\Gamma\curvearrowright\GR$ has infinite orbits.
\item $\Gamma\curvearrowright\Gr$ disconnects the finite sets if and only if $\Gamma\curvearrowright\GR$ disconnects the finite sets.
\item $\Gamma\curvearrowright\Gr$ is non-singular if and only if $\Gamma\curvearrowright\GR$ is non-singular.
\item Let $\Sigma<\Gamma$ be a subgroup. $\Sigma$ is highly core-free w.r.t. $\Gamma\curvearrowright\Gr$ if and only if $\Sigma$ is highly core-free w.r.t. $\Gamma\curvearrowright\GR$.
\item If $\Gamma\curvearrowright\Gr$ is free, $\Gr$ is infinite and $\Gamma$ is torsion free then $\Gamma\curvearrowright\GR$ is free.
\item If $\Gamma\curvearrowright \mathcal{G}$ is strongly faithful, $\Gr$ is infinite then $\Gamma\curvearrowright\GR$ has property $(F )$.
\end{enumerate}
\end{corollary}

\begin{proof}
The assertions $(1)$ to $(6)$ follow directly follows from Propositions \ref{PropRandomExt}, \ref{PropLim} and \ref{PropGraphLim}. Let us prove $(7)$. We recall that $\GR=\Gr_\infty=\cup^\uparrow\Gr_n$, where $\Gr_0=\Gr$ and $\Gr_{n+1}=\widetilde{\Gr_n}$. Let $F\subset V(\Gr_\infty)$ and $S=\{g_1,\dots,g_n\}\subset\Gamma\setminus \{1\}$ be finite subsets. Since $\Gamma\curvearrowright\Gr$ is strongly faithful we deduce, by Proposition \ref{PropRandomExt} (6) and induction that $\Gamma\curvearrowright\Gr_N$ is strongly faithful for all $N$. Let $N\in\N$ be large enough so that $F\subset V(\Gr_N)$. We can use strong faithfulness (and the fact that $\Gr_N$ is infinite) to construct, by induction, pairwise distinct vertices $y_1,\dots y_n\in V(\Gr_N)\setminus F$ such that $y_i\neq g_jy_j$ for all $1\leq i,j\leq n$. 
Define $x:=\{y_1,\dots y_n\}\in \mathcal{P}_f(V(\Gr_N))\subset V(\Gr_{N+1})\subset V(\Gr_\infty)=V(\GR)$. By construction $x\in V(\GR)\setminus F$ and $g_kx\neq x$ for all $k$. Since the inclusion $\Gr_{N+1}\subset\Gr_\infty=\GR$ is open, we also have $x\nsim u$ for all $u\in F$.
\end{proof}

\begin{corollary}\label{CorActR}
Every infinite countable group $\Gamma$ admits an action $\Gamma\curvearrowright\GR$ that is non-singular, has property $(F )$ and disconnects the finite sets. If $\Gamma$ is torsion-free then the action can be chosen to be moreover free.
\end{corollary}

\begin{proof}
Consider the graph $\Gr$ defined by $V(\Gr)=\Gamma$ and $E(\Gr)=\emptyset$ with the action $\Gamma\curvearrowright\Gr$ given by left multiplication which is free, has infinite orbits and hence disconnects the finite sets and is non-singular since $E(\Gr)=\emptyset$. By Corollary \ref{Cor-ActionR} the induced action $\Gamma\curvearrowright\Rc$ has the required properties. 
\end{proof}

\subsection{The Random Extension with paramater}

\noindent We now describe a parametrized version of the induced action. As explained before, freeness is not preserved when passing to the action on the Random Extension and also when passing to the induced action. However, it is easy to compute explicitly the fixed points of any group element in the Random Extension. Now, given an action $\Gamma\curvearrowright\Gr$ and some fixed group elements in $\Gamma$, one can modify the Random Extension by removing the fixed points of our given group elements to make them act freely on this modified version of the Random Extension. Then, the inductive limit process associated to this modified Random Extension will also produce, in some cases, the Random Graph with an action on it, for which our given group elements act freely. We call this process the \textit{parametrized Random Extension}.

\vspace{0.2cm}

\noindent Let $\pi\,:\,\Gamma\curvearrowright \Gr$ be an action of the group $\Gamma$ on the graph $\Gr$ and $F\subset\Gamma$ be a subset. From the action $\pi$, we have a canonical action $\Gamma\curvearrowright\widetilde{\Gr}$ on the Random Extension.

\begin{remark}\label{RmkFree} If $\Gr$ is infinite, for any $g\in\Gamma$, the set of fixed points of $g$ for the action on the Random Extension is either empty or of the form :
$${\rm Fix}_{\widetilde{\Gr}}(g)=\{A\in P_f(V(\Gr))\text{ of the form }A=\sqcup_{i=1}^N\langle g\rangle x_i\text{ with }\langle g\rangle x_i\text{ finite }\forall i\}\sqcup{\rm Fix}_{\Gr}(g).$$
In particular, when the action $\pi$ is free one has ${\rm Fix}_{\widetilde{\Gr}}(g)=\emptyset$ whenever $g$ has infinite order and, if $g$ has finite order, then ${\rm Fix}_{\widetilde{\Gr}}(g)$ is the set of finite unions of $\langle g\rangle$-orbits.
\end{remark}

\noindent Consider the induced subgraph on $V(\widetilde{\Gr})\setminus \{{\rm Fix}_{\widetilde{\Gr}}(g)\,:\,g\in F\}$. Note that if $gFg^{-1}=F$ for all $g\in\Gamma$ then, since $g{\rm Fix}_{\widetilde{\Gr}}(h)={\rm Fix}_{\widetilde{\Gr}}(ghg^{-1})$, the induced subgraph on $V(\widetilde{\Gr})\setminus \{{\rm Fix}_{\widetilde{\Gr}}(g)\,:\,g\in F\}$ is globally $\Gamma$-invariant and we get an action of $\Gamma$ on it by restriction for which the elements of $F$ act freely by construction. However, for a general $F$, we cannot restrict the action and this is why we will remove more sets then the fixed points of elements of $F$.

\vspace{0.2cm}

\noindent Assume from now that $\Gr$ is a graph and $l\in\N^*$. We define the graph $\widetilde{\Gr}_{l}$, the \textit{Random Extension of} $\Gr$ \textit{with parameter $l$}, as the induced subgraph on
$$V(\widetilde{\Gr}_l):=V(\Gr)\sqcup\{U\in\mathcal{P}_f(V(\Gr))\,:\,\gcd(l,\vert U\vert)=1\}\subset V(\widetilde{\Gr}).$$
Fix an action $\Gamma\curvearrowright \Gr$. Since for any $g\in\Gamma$ and any finite subset $U\subset V(\Gr)$ one has $\vert gU\vert=\vert U\vert$, the subgraph $\widetilde{\Gr}_{l}$ is globally $\Gamma$-invariant and we get an action $\widetilde{\pi}_l\,:\,\Gamma\curvearrowright\widetilde{\Gr}_{l}$ by restriction. Note that for any $u\in V(\Gr)$ one has $\{u\}\in V(\widetilde{\Gr}_{l})$. It is clear that the inclusions of $\Gr$ in $\widetilde{\Gr}_{l}$ and of $\widetilde{\Gr}_{l}$ in $\widetilde{\Gr}$ are $\Gamma$-equivariant open (and injective) graph homomorphisms.

\begin{proposition}\label{PropRelRandomExt}
Let $\pi\,:\,\Gamma\curvearrowright \Gr$ be an action and $l\in\N^*$.
\begin{enumerate}
\item $\pi$ is faithful if and only if $\widetilde{\pi}_l$ is faithful.
\item $\pi$ is non-singular if and only if $\widetilde{\pi}_l$ is non-singular.
\item $\pi$ has infinite orbits if and only if $\widetilde{\pi}_l$ has infinite orbits.
\item $\pi$ disconnects the finite sets if and only if $\widetilde{\pi}_l$ disconnects the finite sets.
\item Let $\Sigma<\Gamma$ be a subgroup. $\Sigma$ is highly core-free w.r.t. $\pi$ if and only if $\Sigma$ is highly core-free w.r.t. $\widetilde{\pi}_l$.
\item If $\pi$ is strongly faithful and $\Gr$ is infinite then $\widetilde{\pi}_l$ is strongly faithful.
\item  Let $\Sigma<\Gamma$ be a finite subgroup. If $\Sigma\curvearrowright\Gr$ is free and $\Gr$ is infinite then $\Sigma\curvearrowright \widetilde{\Gr}_{\vert\Sigma\vert}$ is free.
\end{enumerate}
\end{proposition}

\begin{proof}
Assertions $(1)$ to $(6)$ are obvious. Let us prove $(7)$. Since any non-trivial element of $\Sigma$ acts freely, it follows from Remark \ref{RmkFree} that any finite subset $U\subset V(\Gr)$ in the set of fixed points of $\sigma\in\Sigma$ is a finite union of sets of the form $\langle\sigma\rangle x$, hence its size is a multiple of the order of $\sigma$ and $U\notin V(\widetilde{\Gr}_{\vert\Sigma\vert})$.
\end{proof}

\noindent We can now construct \textit{the induced action on $\GR$ of the action $\pi$ with paramter $l$}. Define the sequence of graphs $\Gr_0=\Gr$ with action $\pi_0=\pi$ of $\Gamma$ on it and, for $n\geq 0$, $\Gr_{n+1}=\widetilde{(\Gr_n)}_{l}$ with the action $\pi_{n+1}=(\widetilde{\pi_n})_l$. Consider the inductive limit $\Gr^l_\infty$ with the inductive limit action $\pi_\infty^l\,:\,\Gamma\curvearrowright \Gr^l_\infty$ on it. 
We list the properties of $\pi_l$ in the next Proposition.

\begin{proposition}\label{PropRelRandExt}
Let $\pi\,:\,\Gamma\curvearrowright \Gr$ be an action on a countable graph $\Gr$ and let $l\in\N^*$.
\begin{enumerate}
\item If $\Gr$ is infinite then $\Gr_\infty^l\simeq\GR$.
\item $\pi$ is faithful if and only if $\pi_\infty^l$ is faithful.
\item $\pi$ has infinite orbits if and only if $\pi_\infty^l$ has infinite orbits.
\item $\pi$ disconnects the finite sets if and only if $\pi_\infty^l$ disconnects the finite sets.
\item Let $\Sigma<\Gamma$ be a subgroup. $\Sigma$ is highly core-free w.r.t. to $\pi$ if and only if $\Sigma$ is highly core-free w.r.t. to $\pi_\infty^l$.
\item $\pi$ is non-singular if and only if $\pi_\infty^l$ is non-singular.
\item Let $\Sigma<\Gamma$ be a finite subgroup. If $\Sigma\curvearrowright\Gr$ is free and $\Gr$ is infinite then $\Sigma\curvearrowright\Gr_\infty^{\vert\Sigma\vert}$ is free.
\item Suppose that $\Gamma$ and $\Gr$ are infinite. If $\pi$ is strongly faithful then $\pi_\infty^l$ has property $(F )$.
\end{enumerate}
\end{proposition}

\begin{proof}
The assertions $(2)$ to $(7)$ follow directly follows from Propositions \ref{PropRelRandomExt}, \ref{PropLim} and \ref{PropGraphLim}.

\vspace{0.2cm}

\noindent$(1)$. Since $\Gr$ is infinite it follows that $\Gr_n$ is infinite for all $n$ and $\Gr_\infty$ is infinite. Hence, it suffices to check that $\Gr_\infty$ has property $(R )$. Let $U,V\subset V(\Gr_\infty)$ be two finite subsets such that $U\cap V=\emptyset$ and let $n\in\N$ such that $U,V\subset V(\Gr_n)$. Since $V(\Gr_n)\setminus V$ is infinite, we may find a finite subset $x\subset V(\Gr_n)\setminus V$ such that $U\subset x$ and $gcd(\vert x\vert,\vert\Sigma\vert)=1$. Hence, $x\in V((\widetilde{\Gr_n})_{\vert\Sigma\vert})\subset V(\Gr_\infty^{\vert\Sigma\vert})$ is such that $x\sim u$ for all $u\in U$ and $x\nsim v$ for all $v\in V$.

\vspace{0.2cm}

\noindent$(8)$. Let $S=\{g_1,\dots,g_n\}\subset\Gamma\setminus \{1\}$ and $F\subset V(\Gr_\infty^l)$ be finite subsets. Taking a larger $S$ if necessary, we may and will assume that $\gcd(n,l)=1$. We repeat the proof of Corollary \ref{Cor-ActionR}, assertion $(7)$ and we get pairwise distinct vertices $y_1,\dots y_n\in V(\Gr_N)$, where $F\subset V(\Gr_N)$ such that the element $x=\{y_1,\dots,y_n\}\in\mathcal{P}_f(V(\Gr_N))\subset\widetilde{\Gr}_N$ satisfies the desired properties. Since $\gcd(\vert x\vert,l)=1$, $x\in V((\widetilde{\Gr_N)}_l)\subset V(\Gr_\infty^l)$. This concludes the proof.
\end{proof}

\noindent Let $\Sigma<\Gamma$ be a finite subgroup of an infinite countable group $\Gamma$. By considering the induced action on $\GR$ with parameter $l=\vert\Sigma\vert$ of the free action by left multiplication $\Gamma\curvearrowright\Gamma$ and view $\Gamma$ as a graph with no edges, we obtain the following Corollary.

\begin{corollary}\label{CorRelActR}
Let $\Sigma<\Gamma$ be a finite subgroup of an infinite countable group $\Gamma$. There exists a non-singular action $\Gamma\curvearrowright\GR$ with property $(F )$ such that the action $\Sigma\curvearrowright\GR$ is free, the action $H\curvearrowright \GR$ of any infinite subgroup $H<\Gamma$ disconnects the finite sets and, for any pair of intermediate subgroups $\Sigma'<H'<\Gamma$, if $\Sigma'$ is a highly core-free in $H'$ then $\Sigma'$ is highly core-free w.r.t. $H'\curvearrowright\GR$.
\end{corollary}

\subsection{Homogeneous actions on the Random Graph}

\begin{proposition}
Let $\Gamma\curvearrowright\GR$ be an homogeneous action. The following holds.
\begin{enumerate}
\item $\Gamma\curvearrowright\GR$ disconnects the finite sets.
\item If $N<\Gamma$ is normal then either $N$ acts trivially or it acts homogeneously.
\end{enumerate}
\end{proposition}

\begin{proof}
$(1)$. Let $F\subset V(\GR)$ be a finite set and write $F=\{u_1,\dots,u_n\}$ where the vertices $u_i$ are pairwise distinct. We shall define inductively pairwise distinct vertices $v_1,\dots v_n\in V(\GR)\setminus F$ such that, for all $1\leq i,j\leq n$, $u_i\sim u_j\Leftrightarrow v_i\sim v_j$ and $u_i\nsim v_j$. Take $v_1\in\GR_{\emptyset,F}$. Suppose that for a given $1\leq l\leq n-1$ we have pairwise distinct vertices $v_1,\dots ,v_l\in V(\GR)\setminus F$ such that, for all $1\leq i,j\leq l$, $u_i\sim u_j\Leftrightarrow v_i\sim v_j$ and $u_i\nsim v_j$. Let $U=\{v_i\,:\,1\leq i\leq l\text{ s.t. }u_i\sim u_{l+1}\}$ and $V=\{v_i\,:\,1\leq i\leq l\text{ s.t. }u_i\nsim u_{l+1}\}$. Then $U$ and $V\sqcup F$ are finite and disjoint. Take $v_{l+1}\in\GR_{U,V\sqcup F}$. Then $v_{l+1}\sim v_i\Leftrightarrow u_{l+1}\sim u_i$ and $v_{l+1}\nsim u_i$ for all $1\leq i\leq l$ and $v_{l+1}\notin F\sqcup\{v_i\,:\,1\leq i\leq l\}$. This concludes the construction of the $v_i$ by induction. By the properties of the $v_i$, the map $\varphi\,:\,\{u_i\,:\,1\leq i\leq n\}\rightarrow\{v_i\,:\,1\leq i\leq n\}$ defined by $\varphi(u_i)=v_i$ is an isomorphism between the induced subgraphs. Since $\Gamma\curvearrowright\GR$ is homogeneous there exists $g\in\Gamma$ such that $gu_i=v_i$ for all $1\leq i\leq n$. It follows that $gF\cap F=\emptyset$ and $gu\nsim u'$ for all $u,u'\in F$.

\vspace{0.2cm}

\noindent$(2)$. Write $\overline{N}$ the closure of the image of $N$ inside $\text{Aut}(\Rc)$. By \cite{Tr85}, the abstract group $\text{Aut}(\Rc)$ is simple and, since $\overline{N}$ is normal in $\autR$ one has either $\overline{N}=\{1\}$ or $\overline{N}=\autR$. 
\end{proof}

\noindent Using the previous Proposition and arguing as in \cite[Corollary 1.6]{MS13} we obtain the following Corollary.

\begin{corollary}\label{Obstruction}
If $\Gamma\in\mathcal{H}_\Rc$ then $\Gamma$ is icc and not solvable.
\end{corollary}

\section{Actions of amalgamated free products on the Random Graph}

\noindent Let $\Gamma_1,\Gamma_2$ be two countable groups with a common finite subgroup $\Sigma$ and define $\Gamma=\Gamma_1\underset{\Sigma}{*}\Gamma_2$. Suppose that we have a faithful action $\Gamma\curvearrowright\GR$ and view $\Gamma<\autR$.
\vspace{.2cm}

\noindent Let $Z:=\{\alpha\in\autR\,:\,\alpha\sigma=\sigma\alpha\,\,\forall\sigma\in\Sigma\}$. Note that $Z$ is a closed subgroup of $\autR$, hence a Polish group. Moreover, for all $\alpha\in Z$, there exists a unique group homomorphism $\pi_{\alpha}\,:\,\Gamma\rightarrow\autR$ such that:
$$\pi_{\alpha}(g)=\left\{\begin{array}{lcl}g&\text{if}&g\in\Gamma_1,\\
\alpha^{-1}g\alpha&\text{if}&g\in\Gamma_2.\end{array}\right.$$

\noindent When $\Sigma$ is trivial, we have $Z=\autR$. In this section we prove the following result.

\begin{theorem}\label{ThmMain} If $\Gamma\curvearrowright\Rc$ is non-singular and has property $(F )$, $\Sigma$ is highly core-free w.r.t. $\Gamma_1,\Gamma_2\curvearrowright\Rc$ and $\Sigma\curvearrowright\Rc$ is free then the set $O=\{\alpha\in Z\,:\,\pi_{\alpha}\text{ is homogeneous and faithful}\}$ is a dense $G_\delta$ in $Z$. In particular, for every countably infinite groups $\Gamma_1,\Gamma_2$ we have $\Gamma_1*\Gamma_2\in\mathcal{H}_\Rc$ and, for any finite highly core-free subgroup $\Sigma<\Gamma_1,\Gamma_2$ we have $\Gamma_1\underset{\Sigma}{*}\Gamma_2\in\mathcal{H}_\Rc$.
\end{theorem}

\begin{proof}
We separate the proof in two lemmas. 

\begin{lemma}\label{LemHomogeneous}
If $\Sigma\curvearrowright\Rc$ is free and non-singular and $\Sigma$ is highly core-free w.r.t. $\Gamma_1,\Gamma_2\curvearrowright\GR$ then the set $U=\{\alpha\in Z\,:\,\pi_{\alpha}\text{ is homogeneous}\}$ is a dense $G_\delta$ in $Z$.
\end{lemma}

\begin{proof}
Since $U=\cap_{\varphi\in P(\GR)}U_\varphi$, where $U_\varphi=\{\alpha\in Z\,:\,\exists g\in\Gamma\text{ such that }\pi_{\alpha}(g)\vert_{d(\varphi)}=\varphi\}$ is obviously open, it suffices to show that $U_\varphi$ is dense for all $\varphi\in P(\GR)$. Let $\varphi\in P(\GR)$, $\alpha\in Z$ and $F\subset V(\GR)$ a finite subset. It suffices to show that there exists $\gamma\in Z$ and $g\in\Gamma$ such that $\gamma\vert_F=\alpha\vert_F$ and $\pi_\gamma(g)\vert_{d(\varphi)}=\varphi$. Since $\Sigma$ is highly core-free w.r.t. $\Gamma_1\curvearrowright\GR$ and $\Sigma F$ and $d(\varphi)$ are both finite sets, there exists $g_1\in\Gamma_1$ such that $g_1d(\varphi)\cap \Sigma F=\emptyset$, $g_1x\nsim u$ and $\sigma g_1x\nsim g_1x'$ for all $u\in \Sigma F$, $x,x'\in d(\varphi)$, $\sigma\in\Sigma\setminus \{1\}$ and $\Sigma g_1x\cap\Sigma g_1x'=\emptyset$ for all $x,x'\in d(\varphi)$ with $x\neq x'$. Since $\Sigma$ is highly core-free w.r.t. $\Gamma_1\curvearrowright\GR$, there exists also $g_2\in\Gamma_1$ such that $g_2^{-1}r(\varphi)\cap(\Sigma F\sqcup \Sigma g_1d(\varphi))=\emptyset$ and $g_2^{-1}y\nsim u$ , $\sigma g_2^{-1}y\nsim g_2^{-1}y'$ for all $u\in \Sigma F\sqcup \Sigma g_1d(\varphi)$, $y,y'\in r(\varphi)$, $\sigma\in\Sigma\setminus \{1\}$ and $\Sigma g_2^{-1}y\cap\Sigma g_2^{-1}y'=\emptyset$ for all $y,y'\in r(\varphi)$ with $y\neq y'$. Define $F'=\alpha(\Sigma F)\sqcup\Sigma\alpha(g_1d(\varphi))$. Since $\Sigma$ is highly core-free w.r.t. $\Gamma_2\curvearrowright\GR$, there exists $h\in\Gamma_2$ such that $h F'\cap F'=\emptyset$ and for all $u,v\in F'$ and $\sigma\in\Sigma\setminus \{1\}$, $hu\nsim v$, $\sigma hu\nsim hv$ and $\Sigma hu\cap\Sigma h u'=\emptyset$ for all $u,u'\in F'$ with $u\neq u'$. Define $A=(\Sigma F)\sqcup\Sigma g_1d(\varphi)\sqcup\Sigma g_2^{-1}r(\varphi)$ and $B=\alpha(\Sigma F)\sqcup\alpha(\Sigma g_1d(\varphi))\sqcup\Sigma h\alpha(g_1d(\varphi))$. Note that $\Sigma A= A$ and, since $\alpha\in Z$, $\Sigma B=B$. Since the $\Sigma$ action is free, we can define a bijection $\gamma_0\,:\,A\rightarrow B$ by $\gamma_0(u)=\alpha(u)$ for $u\in\Sigma F\sqcup\Sigma g_1d(\varphi)$ and $\gamma_0(\sigma g_2^{-1}\varphi(x))=\sigma h\alpha g_1x$, for all $x\in d(\varphi)$ and $\sigma\in\Sigma$. By construction, $\gamma_0$ is a partial isomorphism and $\gamma_0\sigma=\sigma\gamma_0$ for all $\sigma\in\Sigma$. By Proposition \ref{PropExtension} there exists an extension $\gamma\in Z$ of $\gamma_0$. Note that $\gamma\vert_F=\alpha\vert_F$ moreover, with $g=g_2hg_1\in \Gamma$ we have, for all $x\in d(\varphi)$, $\pi_\gamma(g)x=g_2\gamma^{-1}h\gamma g_1x=g_2\gamma^{-1}h\alpha(g_1x)=g_2g_2^{-1}\varphi(x)=\varphi(x)$.
\end{proof}

\begin{lemma}\label{LemFaithful}
If $\Sigma\curvearrowright\Rc$ is free, $\Gamma\curvearrowright \mathcal{R}$ is non-singular and has property $(F )$ then the set $V=\{\alpha\in Z\,:\,\pi_{\alpha}\text{ is faithful}\}$ is a dense $G_\delta$ in $Z$.
\end{lemma}

\begin{proof}
Writing $V=\cap_{g\in\Gamma\setminus \{1\}}V_g$, where $V_g=\{\alpha\in Z\,:\,\pi_{\alpha}(g)\neq\id\}$ is obviously open, it suffices to show that $V_g$ is dense for all $g\in\Gamma\setminus \{1\}$. If $g\in \Gamma_k\setminus \{1\}$ ($k=1,2$) then it is easy to see that $V_g=Z$. Hence it suffices to show that $V_g$ is dense for all $g$ reduced of length at least $2$. Write $g=g_{i_n}\dots g_{i_1}$, where $n\geq 2$ and $g_{i_k}\in\Gamma_{i_k}\setminus \Sigma$ reduced expression for $g$. Fix $\alpha\in Z$ and $F\subset V(\Rc)$ a finite subset. Define the finite sets
$$\widetilde{F}:=\Sigma F\cup\alpha(\Sigma F)\cup\left(\bigcup_{l=1}^n(g_{i_l}\dots g_{i_1})^{-1}\left(\Sigma F\cup\alpha(\Sigma F)\right)\right)\subset V(\Rc),$$
and
$$S:=\{\sigma g_{i_l}\dots g_{i_1}\,:\,1\leq l\leq n,\sigma\in\Sigma\}\cup\{g_{i_1}^{-1}\dots g_{i_l}^{-1}\sigma g_{i_k}\dots g_{i_1}\,:\,1\leq l< k\leq n,\sigma\in\Sigma\}\subset\Gamma\setminus\{1\}.$$
Using property $(F)$ for $\Gamma\curvearrowright\Rc$ we find a vertex $x\in V(\Rc)\setminus \widetilde{F}$ such that $x\nsim u$ for all $u\in\widetilde{F}$ and $gx\neq x$ for all $g\in S$. In particular, the sets $\Sigma F$ (resp. $\alpha(\Sigma F)$) and $\Sigma x$, $\Sigma g_{i_l}\dots g_{i_1}x$ for $1\leq l\leq n$ are pairwise distincts and, since $\Gamma\curvearrowright\Rc$ is non-singular, the only vertices on the induced subgraph on $Y:=\Sigma F\sqcup(\bigsqcup _{l=1}^n\Sigma g_{i_l}\dots g_{i_1}x)\sqcup\Sigma x$ (resp. $Y':=\Sigma\alpha(F)\sqcup(\bigsqcup _{l=1}^n\Sigma g_{i_l}\dots g_{i_1}x)\sqcup\Sigma x$) are the ones with extremities in $\Sigma F$ (resp. $\Sigma\alpha(F)$). Hence, the bijection $\gamma_0\,:\, Y\rightarrow Y'$ defined by $\gamma_0\vert_{\Sigma F}=\alpha\vert_{\Sigma F}$ and $\gamma_0\vert_{Y\setminus\Sigma F}=\id$ is a graph isomorphism between the induced subgraphs. By construction and since $\alpha\in Z$ we have $\gamma_0\sigma=\sigma\gamma_0$ for all $\sigma\in\Sigma$. By Proposition \ref{PropExtension} there exists an extension $\gamma\in Z$ of $\gamma_0$. Then $\gamma\vert_F=\alpha\vert_F$ and $\pi_\gamma(g_{i_n}\dots g_{i_1})x=g_{i_n}\dots g_{i_1}x\neq x$, since $g_{i_n}\dots g_{i_1}\in S$. Hence $\gamma\in V_g$.
\end{proof}

\noindent\textit{End of the proof of the Theorem.} 
The first part of the Theorem follows from Lemmas \ref{LemHomogeneous} and \ref{LemFaithful} since $O=U\cap V$. The last part follows from the first part and Corollary \ref{CorRelActR}.
\end{proof}

\noindent We have a similar result when only one of the factors in the free product is infinite.

\begin{theorem}\label{ThmOneFiniteFactor}
Suppose that $\Gamma_2$ is finite such that $[\Gamma_2:\Sigma]\geq 2$. If $\Gamma\curvearrowright\Rc$ is non-singular and has property $(F )$, $\Sigma$ is highly core-free w.r.t. $\Gamma_1\curvearrowright\Rc$ and the action $\Gamma_2\curvearrowright\Rc$ is free then the set $$O=\{\alpha\in Z\,:\,\pi_{\alpha}\text{ is homogeneous and faithful}\}$$
is a dense $G_\delta$ in $Z$. In particular, for every countably infinite group $\Gamma_1$, for every finite non-trivial group $\Gamma_2$ we have $\Gamma_1*\Gamma_2\in\mathcal{H}_\Rc$ and, for any common finite subgroup $\Sigma<\Gamma_1,\Gamma_2$ such that $\Sigma$ is highly core-free in $\Gamma_1$ and $[\Gamma_2:\Sigma]\geq 2$ we have $\Gamma_1\underset{\Sigma}{*}\Gamma_2\in\mathcal{H}_\Rc$.
\end{theorem}

\begin{proof}
We first prove the analogue of Lemma \ref{LemHomogeneous}.

\begin{lemma}\label{LemHomogeneousOneFiniteFactor}
If $\Sigma$ is highly core-free w.r.t. $\Gamma_1\curvearrowright\GR$, $\Gamma_2$ is finite and $\Gamma_2\curvearrowright\GR$ is free and non-singular such that $ [\Gamma_2:\Sigma ] \geq 2$, then the set $U=\{\alpha\in Z\,:\,\pi_{\alpha}\text{ is homogeneous}\}$ is a dense $G_\delta$ in $Z$.
\end{lemma}

\begin{proof}
Since $U=\cap_{\varphi\in P(\GR)}U_\varphi$, where $U_\varphi=\{\alpha\in Z\,:\,\exists g\in\Gamma\text{ such that }\pi_{\alpha}(g)\vert_{d(\varphi)}=\varphi\}$ is obviously open, it suffices to show that $U_\varphi$ is dense for all $\varphi\in P(\GR)$. Let $\varphi\in P(\GR)$, $\alpha\in Z$ and $F\subset V(\GR)$ a finite subset. Since $\Sigma$ is highly core-free w.r.t. $\Gamma_1\curvearrowright\GR$, there exists $g_1\in\Gamma_1$ such that $g_1d(\varphi)\cap \Sigma F=\emptyset$, $g_1x\nsim u$ and $\sigma g_1x\nsim g_1x'$ for all $u\in \Sigma F$, $x,x'\in d(\varphi)$, $\sigma\in\Sigma\setminus \{1\}$ and $\Sigma g_1x\cap\Sigma g_1x'=\emptyset$ for all $x,x'\in d(\varphi)$ with $x\neq x'$. There exists also $g_2\in\Gamma_1$ such that $g_2^{-1}r(\varphi)\cap(\Sigma F\sqcup \Sigma g_1d(\varphi))=\emptyset$ and $g_2^{-1}y\nsim u$ , $\sigma g_2^{-1}y\nsim g_2^{-1}y'$ for all $u\in \Sigma F\sqcup \Sigma g_1d(\varphi)$, $y,y'\in r(\varphi)$, $\sigma\in\Sigma\setminus \{1\}$ and $\Sigma g_2^{-1}y\cap\Sigma g_2^{-1}y'=\emptyset$ for all $y,y'\in r(\varphi)$ with $y\neq y'$. In the sequel we write $d(\varphi)=\{x_1,\dots,x_n\}$.

\vspace{0.2cm}

\noindent \textbf{Claim.}\textit{ There exists pairwise distinct vertices $z_1,\dots,z_n\in V(\GR)\setminus \Gamma_2\alpha(F)$ such that
\begin{itemize}
\item $\Gamma_2 z_i\cap \Gamma_2 z_j=\emptyset$, $\forall 1\leq i\neq j\leq n$;
\item $\Sigma hz_i\cap\Sigma z_i=\emptyset$, $\forall h\in\Gamma_2\setminus\Sigma$, $\forall$ $1\leq i\leq n$;
\item $hz_i\nsim z_j$, $\forall h\in \Gamma_2\setminus \{1\}$, $\forall 1\leq i,j\leq n$;
\item $x_i\sim x_j\Leftrightarrow z_i\sim z_j$, $\forall 1\leq i,j\leq n$;
\item $z_i\nsim u$, $\forall u\in \Gamma_2 \alpha(F)$, $\forall 1\leq i\leq n$.  
\end{itemize}}

\noindent\textit{Proof of the Claim.} We define inductively the vertices $z_1,\dots,z_n\in V(\GR)\setminus \Gamma_2\alpha(F)$. For $n=1$ we take $z_1\in\mathcal{R}_{\emptyset,\Gamma_2\alpha(F)}\setminus \Gamma_2\alpha(F)$. Then $z_1\notin \Gamma_2\alpha(F)$ and, for all $u\in \Gamma_2\alpha(F)$ and for all $h\in \Gamma_2$, $u\nsim z_1$ and $hz_1\nsim z_1$ since $\Gamma_2\curvearrowright\GR$ is non-singular. Moreover, since $\Gamma_2\curvearrowright\Rc$ is free, we have $hz_1\neq\sigma z_i$ for all $h\in\Gamma_2\setminus \Sigma$ and all $\sigma\in\Sigma$.
\vspace{0.2cm}

\noindent Suppose that, for a given $1\leq l\leq n-1$, we have pairwise distinct vertices $z_1,\dots,z_l\in V(\GR)\setminus \Gamma_2\alpha(F)$ such that $\Gamma_2 z_i\cap \Gamma_2 z_j=\emptyset$, $\forall 1\leq i\neq j\leq l$, $\Sigma hz_i\cap\Sigma z_i=\emptyset$, $\forall$ $h\in\Gamma_2\setminus \Sigma$, $\forall$ $1\leq i\leq l$, $hz_i\nsim z_j$, $\forall h\in \Gamma_2\setminus \{1\}$, $x_i\sim x_j\Leftrightarrow z_i\sim z_j$ and $z_i\nsim u$, $\forall u\in \Gamma_2 \alpha(F)$ $\forall 1\leq i,j\leq l$. Define 
$$U=\{z_i\,:\,1\leq i\leq l\text{ s.t. }x_i\sim x_{l+1}\}$$and 
$$V=\{z_i\,:\,1\leq i\leq l\text{ s.t. }x_i\nsim x_{l+1}\}\sqcup \{hz_i :\,1\leq i\leq l, h\in \Gamma_2\setminus\{1\} \}\sqcup \Gamma_2\alpha( F).$$
Since $\mathcal{R}_{U,V}$ is infinite and $\Gamma_2$ is finite, we may take
$$z_{l+1}\in\GR_{U,V}\setminus \left(\Gamma_2 \alpha( F)\sqcup\sqcup_{i=1}^l\Gamma_2 z_i\right).$$
Then, for all $u\in \Gamma_2\alpha(F)$, $u\nsim z_{l+1}$ and, for all $1\leq i\leq l$, $h\in\Gamma_2\setminus\{1\}$, $x_i\sim x_{l+1}\Leftrightarrow z_i\sim z_{l+1}$, $h z_i\nsim z_{l+1}$ and, for all $i\neq l+1$, $\Gamma_2 z_i\cap\Gamma_2 z_{l+1}=\emptyset$. Also, since $\Gamma_2\curvearrowright\Rc$ is free, we have $\Sigma hz_{l+1}\cap\Sigma z_{l+1}=\emptyset$ for all $h\in\Gamma_2\setminus \Sigma$. This completes the proof.

\vspace{0.2cm}

\noindent\textit{End of the proof of the Lemma.} Write $d(\varphi)=\{x_1,\dots,x_n\}$ and define, for $1\leq k\leq n$, $y_k=\varphi(x_k)$. Let $z_1,\dots z_n$ be the elements obtained by the Claim. Take $h\in \Gamma_2\setminus\Sigma$. Then the sets $\alpha(\Sigma F)$, $\Sigma z_i$ for $1\leq i\leq n$, and $\Sigma hz_i$ for $1\leq i\leq n$ are pairwise disjoint. Moreover, $u\nsim \sigma z_i$, $u\nsim \sigma h z_i$ and $\sigma' z_i\nsim \sigma h z_j$ for all $u\in \alpha(\Sigma F)$, for all $\sigma, \sigma'\in \Sigma$ and for all $1\leq i, j\leq n$. Define 
$$A=\Sigma F\sqcup \left(\sqcup_{i=1}^n\Sigma g_1x_i\right)\sqcup\left(\sqcup_{i=1}^n\Sigma g_2^{-1}y_i\right)\text{ and }B=\alpha(\Sigma F)\sqcup\left( \sqcup_{i=1}^n\Sigma z_i\right)\sqcup\left( \sqcup_{i=1}^n\Sigma hz_i\right)$$
 and consider the induced graph structure on $A$ and $B$. Note that $\Sigma A= A$ , $\Sigma B=B$ and the only vertices in $A$ (resp. $B$) are the ones with extremities in $\Sigma F$ (resp. $\alpha(\Sigma F)$). Since $\Sigma\curvearrowright\Rc$ is free, we may define a bijection $\gamma_0\,:\,A\rightarrow B$ by $\gamma_0(u)=\alpha(u)$ for $u\in \Sigma F$ and $\gamma_0(\sigma g_1x_i)=\sigma z_i$, $\gamma_0(\sigma g_2^{-1}y_i)=\sigma hz_i$ for all $1\leq i\leq n$ and for all $\sigma\in\Sigma$ which is a graph isomorphism satisfying $\gamma_0\sigma=\sigma\gamma_0$ for all $\sigma\in\Sigma$. By Proposition \ref{PropExtension}, there exists an extension $\gamma\in Z$ of $\gamma_0$. Then $\gamma\vert_F=\alpha\vert_F$ and with $g=g_2hg_1\in \Gamma$ we have, for all $1\leq i\leq n$,
$$\pi_\gamma(g)x_i=g_2\gamma^{-1}h\gamma g_1x_i=g_2\gamma^{-1}hz_i=g_2g_2^{-1}y_i=y_i.$$
\end{proof}

\noindent\textit{End of the proof of the Theorem \ref{ThmOneFiniteFactor}.} 
The first assertion of the Theorem follows from Lemmas \ref{LemHomogeneousOneFiniteFactor} and \ref{LemFaithful}. The last part follows from the first part and Corollary \ref{CorRelActR}, where the group $\Sigma$ in Corollary \ref{CorRelActR} is actually our group $\Gamma_2$, the group $\Sigma'$ is our group $\Sigma$ and the group $H'$ is our group $\Gamma_1$.\end{proof}

\section{Actions of HNN extensions on the Random Graph}

\noindent Let $\Sigma<H$ be a finite subgroup of a countable group $H$ and $\theta\,:\,\Sigma\rightarrow H$ be an injective group homomorphism. Define $\Gamma={\rm HNN}(H,\Sigma,\theta)$ the HNN-extension and let $t\in\Gamma$ be the ``stable letter'' i.e. $\Gamma$ is the universal group generated by $\Sigma$ and $t$ with the relations $t\sigma t^{-1}=\theta(\sigma)$ for all $\sigma\in\Sigma$. For $\epsilon\in\{-1,1\}$, we write
$$\Sigma_\epsilon:=\left\{\begin{array}{lcl}\Sigma&\text{if}&\epsilon=1,\\ \theta(\Sigma)&\text{if}&\epsilon=-1.\end{array}\right.$$

\noindent Suppose that we have a faithful action $\Gamma\curvearrowright\Rc$ and view $\Gamma<\autR$. Define the closed (hence Polish space) subset $Z=\{\alpha\in\autR\,:\,\theta(\sigma)=\alpha\sigma\alpha^{-1}\text{ for all }\sigma\in\Sigma\}\subset\autR$ and note that it is non-empty (since $t\in Z$). By the universal property of $\Gamma$, for each $\alpha\in Z$ there exists a unique group homomorphism $\pi_\alpha\,:\,\Gamma\rightarrow\autR$ such that
$$\pi_\alpha\vert_H=\id_H\text{ and }\pi_\alpha(t)=\alpha.$$

\noindent In this section we prove the following result.

\begin{theorem}\label{ThmHNN}
If $\Gamma\curvearrowright\GR$ is non-singular, has property $(F )$, $\Sigma_\epsilon\curvearrowright\Rc$ is free and $\Sigma_\epsilon$ is highly core-free w.r.t. $H\curvearrowright\GR$ for all $\epsilon\in\{-1,1\}$ then the set
$$O=\{\alpha\in\autR\,:\,\pi_\alpha\text{ is faithful and homogeneous}\}$$
is a dense $G_\delta$ in $Z$. In particular, for any finite subgroup $\Sigma$ of an infinite countable group $H$ such that $\Sigma_\epsilon<H$ is highly core-free for all $\epsilon\in\{-1,1\}$, we have $\text{HNN}(H,\Sigma,\theta)\in\mathcal{H}_\Rc$.
\end{theorem}

\noindent We separate the proof in two lemmas.

\begin{lemma}\label{LemHNNHomogeneous}
If, for all $\epsilon\in\{-1,1\}$, $\Sigma_\epsilon\curvearrowright\Rc$ is free, non-singular and $\Sigma_\epsilon$ is highly core-free w.r.t. $H\curvearrowright\GR$ then the set $U=\{\alpha\in\autR\,:\,\pi_\alpha\text{ is homogeneous}\}$ is a dense $G_\delta$ in $Z$.
\end{lemma}

\begin{proof}
Since $U=\cap_{\varphi\in P(\GR)}U_\varphi$, where $U_\varphi=\{\alpha\in Z\,:\,\exists g\in\Gamma\text{ such that }\pi_{\alpha}(g)\vert_{d(\varphi)}=\varphi\}$ is obviously open, it suffices to show that $U_\varphi$ is dense for all $\varphi\in P(\GR)$. Let $\varphi\in P(\GR)$, $\alpha\in Z$ and $F\subset V(\GR)$ a finite subset. It suffices to show that there exists $\gamma\in Z$ and $g\in\Gamma$ such that $\gamma\vert_F=\alpha\vert_F$ and $\pi_\gamma(g)\vert_{d(\varphi)}=\varphi$. Since $\Sigma$ is highly core-free w.r.t. $H\curvearrowright\GR$, there exists $g_1\in H$ such that $g_1d(\varphi)\cap \Sigma F=\emptyset$ and $g_1x\nsim u$, $\sigma g_1x\nsim g_1x'$ for all $u\in \Sigma F$, $x,x'\in d(\varphi)$, $\sigma\in\Sigma\setminus \{1\}$ and $\Sigma g_1x\cap\Sigma g_1x'=\emptyset$ for all $x,x'\in d(\varphi)$ with $x\neq x'$. Since $\theta(\Sigma)$ is highly core-free w.r.t. $H\curvearrowright\GR$, there exists $g_2\in H$ such that $g_2^{-1}r(\varphi)\cap(\alpha(\Sigma F\sqcup \Sigma g_1d(\varphi)))=\emptyset$ and $g_2^{-1}y\nsim u$, $\theta(\sigma) g_2^{-1}y\nsim g_2^{-1}y'$ for all $u\in\alpha(\Sigma F\sqcup \Sigma g_1d(\varphi))$, $y,y'\in r(\varphi)$, $\sigma\in\Sigma\setminus \{1\}$ and $\theta(\Sigma) g_2^{-1}y\cap\theta(\Sigma) g_2^{-1}y'=\emptyset$ for all $y,y'\in r(\varphi)$ with $y\neq y'$. In the sequel we write $d(\varphi)=\{x_1,\dots,x_n\}$ and $y_i=\varphi(x_i)$. Define $Y=(\sqcup_{i=1}^n\Sigma g_1 x_i)\sqcup(\sqcup_{i=1}^n\Sigma\alpha^{-1}(g_2^{-1}y_i))\sqcup(\Sigma F)$ and note that $\alpha(Y)=(\sqcup_{i=1}^n\theta(\Sigma) \alpha(g_1 x_i))\sqcup(\sqcup_{i=1}^n\theta(\Sigma)(g_2^{-1}y_i))\sqcup(\theta(\Sigma)\alpha( F))$. Note that $\Sigma Y=Y$ and $\theta(\Sigma)\alpha(Y)=\alpha(Y)$. Define the bijection $\gamma_0\,:\, Y\rightarrow\alpha(Y)$ by $\gamma_0(\sigma g_1x_i)=\theta(\sigma) g_2^{-1}y_i$, $\gamma_0(\sigma\alpha^{-1}(g_2^{-1}y_i))=\theta(\sigma)\alpha(g_1x_i)$ for all $1\leq i\leq n$, $\sigma\in\Sigma$ and $\gamma_0\vert_{\Sigma F}=\alpha\vert_{\Sigma F}$. By construction $\gamma_0$ is a graph isomorphism such that $\gamma_0\sigma=\theta(\sigma)\gamma_0$ for all $\sigma\in\Sigma$. By Proposition \ref{PropExtension} there exists an extension $\gamma\in Z$ of $\gamma_0$. Note that $\gamma\vert_F=\alpha\vert_F$ moreover, with $g=g_2tg_1\in \Gamma$ we have, for all $1\leq i\leq n$, $\pi_\gamma(g)x_i=g_2\gamma g_1x_i=g_2\gamma_0(g_1x_i)=g_2g_2^{-1}y_i=y_i$.
\end{proof}

\begin{lemma}\label{LemHNNfaithful}
If, for all $\epsilon\in\{-1,1\}$, $\Sigma_\epsilon\curvearrowright\Rc$ is free and $\Gamma\curvearrowright\Rc$ is non-singular and has property $(F)$ then the set $V=\{\pi_\alpha\,:\,\pi_\alpha\text{ is faithful}\}$ is a dense $G_\delta$ in $Z$.
\end{lemma}

\begin{proof}
Since $V=\bigcap_{g\in\Gamma\setminus \{1\}}V_g$, where $V_g=\{\alpha\in Z\,:\,\pi_\alpha(g)\neq\id\}$ is clearly open, it suffices to show that $V_g$ is dense for all $g\in\Gamma\setminus \{1\}$. We may and will assume that $g\notin H$, since when $g\in H$ we have $V_g=Z$. Write $g=h_nt^{\epsilon_n}\dots t^{\epsilon_1}h_0$ a reduced expression for $g$, where $n\geq 1$, $h_k\in H$ and $\epsilon_k\in\{-1,1\}$, $\forall 1\leq k\leq n$. Fix $\alpha\in Z$ and $F\subset V(\Rc)$ be a finite subset. 
\vspace{0.2cm}

\noindent Define 
$$H_1=\left\{\begin{array}{lcl}\Sigma h_0&\text{if}&\epsilon_1=1\\ \Sigma t^{-1} h_0&\text{if}&\epsilon_1=-1\end{array}\right.\text{ and } \widetilde{H}_1=\left\{\begin{array}{lcl}\theta(\Sigma) th_0&\text{if}&\epsilon_1=1\\ \theta(\Sigma) h_0&\text{if}&\epsilon_1=-1.\end{array}\right.$$ 
For $2\leq l\leq n$, define 
$$H_l=\left\{\begin{array}{lcl}\Sigma h_{l-1}t^{\epsilon_{l-1}}\dots t^{\epsilon_1}h_0&\text{if}&\epsilon_l=1\\ \Sigma t^{-1}h_{l-1}t^{\epsilon_{l-1}}\dots t^{\epsilon_1}h_0&\text{if}&\epsilon_l=-1\end{array}\right.\text{and }\widetilde{H}_l=\left\{\begin{array}{lcl}\theta(\Sigma)t h_{l-1}t^{\epsilon_{l-1}}\dots t^{\epsilon_1}h_0&\text{if}&\epsilon_l=1\\ \theta(\Sigma)h_{l-1}t^{\epsilon_{l-1}}\dots t^{\epsilon_1}h_0&\text{if}&\epsilon_l=-1.\end{array}\right.$$
Let $G:=\cup_{l=1}^nH_l$, $\widetilde{G}:=\cup_{l=1}^n\widetilde{H}_l$ and $S:=G^{-1}G\cup\widetilde{G}^{-1}\widetilde{G}\cup\{g\}\subset\Gamma$ be finite subsets of $\Gamma$ and $\widetilde{F}:=\cup_{s\in G\cup\widetilde{G}}s^{-1}(\Sigma F\cup\alpha(\Sigma F))\subset V(\Rc)$ a finite subset of $V(\Rc)$. Since $\Gamma\curvearrowright\Rc$ has property $(F)$ there exists $x\in V(\Rc)\setminus \widetilde{F}$ such that $x\nsim u$ for all $u\in\widetilde{F}$ and $sx\neq x$ for all $s\in S\setminus \{1\}$. In particular, defining $Y_l:=H_lx$ and $\widetilde{Y}_l:=\widetilde{H}_lx$, the sets $\Sigma F$ and $Y_l$ (resp. $\alpha(\Sigma F)$ and $\widetilde{Y}_l$) for $1\leq l\leq n$ are pairwise disjoint and, using moreover the fact that $\Gamma\curvearrowright\Rc$ is non-singular, the only vertices on the induced subgraph on $Y:=\Sigma F\sqcup(\bigsqcup_{l=1}^n Y_l)$ (resp. on $\widetilde{Y}:=\theta(\Sigma)\alpha( F)\sqcup(\bigsqcup_{l=1}^n \widetilde{Y}_l)$) are the ones with extremities in $\Sigma F$ (resp. in $\theta(\Sigma)\alpha(F)$). It implies that the bijection $\gamma_0\,:\, Y\rightarrow\widetilde{Y}$ defined by $\gamma_0\vert_{\Sigma F}=\alpha\vert_{\Sigma F}$ and, for all $1\leq l\leq n$, $\gamma_0\vert_{Y_l}=t\vert_{Y_l}$ is actually a graph isomorphism between the induced subgraphs. Since we clearly have $\gamma_0\sigma=\theta(\sigma)\gamma_0$ for all $\sigma\in\Sigma$, there exists, by Proposition \ref{PropExtension} an extension $\gamma\in Z$ of $\gamma_0$. Then $\gamma$ satisfies $\gamma\vert _F=\alpha\vert_F$ and $\pi_\gamma(g)x=h_n\gamma^{\epsilon_n}\dots\gamma^{\epsilon_1}h_0x=h_nt^{\epsilon_n}\dots t^{\epsilon_1}h_0x=gx\neq x$ since $g\in S$. It follows that $\gamma\in V_g$.
\end{proof}

\noindent\textit{End of the proof of Theorem \ref{ThmHNN}.} The first assertion follows directly from Lemmas \ref{LemHNNHomogeneous} and \ref{LemHNNfaithful} since $O=U\cap V$ and the last part follows from the first part and Corollary \ref{CorRelActR}.

\section{Actions of groups acting on trees on the random graph}

\noindent Let $\Gamma$ be a group acting without inversion on a non-trivial tree. By \cite{Se83}, the quotient graph $\mathcal{G}$ can be equipped with the structure of a graph of groups $(\mathcal{G},\{\Gamma_p\}_{p\in\VG},\{\Sigma_e\}_{e\in\EG})$ where each $\Sigma_e=\Sigma_{\overline{e}}$ is isomorphic to an edge stabilizer and each $\Gamma_p$ is isomorphic to a vertex stabilizer and such that $\Gamma$ is isomorphic to the fundamental group $\pi_1(\Gamma, \mathcal{G})$ of this graph of groups i.e., given a fixed maximal subtree $\mathcal{T}\subset\mathcal{G}$, the group $\Gamma$ is generated by the groups $\Gamma_p$ for $p\in\VG$ and the edges $e\in\EG$ with the relations
$$\overline{e}=e^{-1},\quad s_{e}(x)=er_{e}(x)e^{-1}\,\,,\,\forall x\in\Sigma_e\quad\text{and}\quad e=1\,\,\,\,\forall e\in {\rm E}(\mathcal{T}),$$
where $s_e\,:\,\Sigma_e\rightarrow \Gamma_{s(e)}$ and $r_e=s_{\overline{e}}\,:\,\Sigma_e\rightarrow\Gamma_{r(e)}$ are respectively the source and range group monomomorphisms.

\begin{theorem}\label{Trees}
If $\Gamma_p$ is countably infinite, for all $p\in\VG$, $\Sigma_e$ is finite and $s_e(\Sigma_e)$ is highly core-free in $\Gamma_{s(e)}$, for all $e\in\EG$, then $\Gamma\in\mathcal{H}_\Rc$.
\end{theorem}

\begin{proof}

Let $e_0$ be one edge of $\mathcal{G}$ and $\mathcal{G}'$ be the graph obtained from $\mathcal{G}$ by removing the edges $e_0$ and $\overline{e_0}$.

\vspace{0.2cm}

\noindent \textbf{Case 1: $\mathcal{G}'$ is connected.} It follows from Bass-Serre theory that $\Gamma={\rm HNN}(H,\Sigma,\theta)$ where $H$ is fundamental group of our graph of groups restricted to $\mathcal{G}'$, $\Sigma=r_{e_0}(\Sigma_{e_0})<H$ is a subgroup and $\theta\,:\,\Sigma\rightarrow H$ is given by $\theta=s_{e_0}\circ r_{e_0}^{-1}$. By hypothesis $H$ is countably infinite, $\Sigma$ is finite and, since $\Sigma<\Gamma_{r(e_0)}$ (resp. $\theta(\Sigma)<\Gamma_{s(e_0)}$) is a highly core-free subgroup, $\Sigma<H$ (resp. $\theta(\Sigma)<H$) is also a highly core-free subgroup. Thus we may apply Theorem \ref{ThmHNN} to conclude that $\Gamma\in\mathcal{H}_\Rc$.

\vspace{0.2cm}

\noindent \textbf{Case 2: $\mathcal{G}'$ is not connected.} Let $\mathcal{G}_1$ and $\mathcal{G}_2$ be the two connected components of $\mathcal{G}'$ such that $s(e_0)\in{\rm V}(\mathcal{G}_1)$ and $r(e_0)\in{\rm V}(\mathcal{G}_2)$. Bass-Serre theory implies that $\Gamma=\Gamma_1*_{\Sigma{e_0}}\Gamma_2$, where $\Gamma_i$ is the fundamental group of our graph of groups restricted to $\mathcal{G}_i$, $i=1,2$, and $\Sigma_{e_0}$ is viewed as a highly core-free subgroup of $\Gamma_1$ via the map $s_{e_0}$ and as a highly core-free subgroup of $\Gamma_2$ via the map $r_{e_0}$ since $s_{e_0}(\Sigma_{e_0})$ is highly core-free in $\Gamma_{s(e_0)}$ and $r_{e_0}(\Sigma_{e_0})$ is highly core-free in $\Gamma_{r(e_0)}$ by hypothesis. Since $\Gamma_1$ and $\Gamma_2$ are countably infinite and $\Sigma_{e_0}$ is finite, we may apply Theorem \ref{ThmMain} to conclude that $\Gamma\in\mathcal{H}_\Rc$.
\end{proof}

\section{Actions of free groups on the random graph}

\noindent Recall that $P(\Rc)$ denotes the set of isomorphisms of $\Rc$ between finite induced subgraphs $d(\varphi),r(\varphi)\subset V(\Rc)$ and note that $P(\Rc)$ has a natural structure of groupoid.
\vspace{0.2cm}

\noindent In this Section, we prove Theorem $D$. The main tool, called \textit{an elementary extension}, is a refinement of the ``back and forth'' method used to extend any partial isomorphism $\varphi \in P(\Rc)$ to an automorphism of $\Rc$.
\vspace{0.2cm}

\subsection{Elementary extensions} 

\noindent Let $\Phi\subset P(\Rc)$ and $F\subset V(\Rc)$ be finite subsets. Let $\gamma\in\Phi$. We construct a partial isomorphism $\tilde \gamma \in P(\Rc)$ which extends $\gamma$ as follows: first, we set
$$
 K :=  F \cup \bigcup_{\varphi\in\Phi} \big( d(\varphi) \cup r(\varphi) \big),\quad D :=  K \setminus d(\gamma), \quad R :=  K \setminus r(\gamma).
$$
Note that $D$ and $R$ have the same cardinality and write $D=\{x_1,\ldots,x_m\}$ and $R=\{v_1,\ldots,v_m\}$. We then successively find vertices $u_1,\ldots,u_m \in V(\Rc)\setminus K$ such that:
\begin{itemize}
 \item the vertex $u_j$ is adjacent to all vertices $u\in d(\gamma)$ such that $v_j\sim \gamma(u)$ and all vertices $u_{j'}$, with $j'<j$, such that $v_{j'}\sim v_j$;
 \item it is not adjacent to any other vertices in $K$.  
\end{itemize}
 After that, we successively find vertices $y_1,\ldots,y_m \in V(\Rc)\setminus (K\cup\{u_1,\ldots,u_m\})$ such that:
\begin{itemize}
 \item the vertex $y_i$ is adjacent to all vertices $y\in r(\gamma)$ such that $x_j\sim \gamma^{-1}(y)$ and all vertices $y_{i'}$, with $i'<i$, such that $x_{i'}\sim x_i$;
 \item it is not adjacent to any other vertices in $K$, nor to vertices $u_1,\ldots,u_m$.  
\end{itemize}
Finally, since the sets $D=\{x_1,\ldots,x_m\}, d(\gamma), \{u_1,\ldots,u_m\}$ are pairwise disjoint, and the sets $\{y_1,\ldots,y_m\}, r(\gamma), R=\{v_1,\ldots,v_m\}$ are pairwise disjoint, we can define a partial bijection
\[
 \tilde\gamma: K \sqcup \{u_1,\ldots,u_m\} \to \{y_1,\ldots,y_m\} \sqcup K
\]
which extends $\gamma$ by setting $\tilde\gamma(x_i)=y_i$ for all $i=1,\ldots,m$, and $\tilde\gamma(z) = \gamma(z)$ for all $z\in d(\gamma)$, and $\tilde\gamma(u_j) = v_j$ for all $j=1,\ldots, n$.

\begin{definition}
 The above $\tilde\gamma$ is called an \textit{elementary extension} of the triple $(\gamma,\Phi,F)$.
\end{definition}
 
\begin{proposition}\label{PropElExt}
 Any elementary extension $\tilde\gamma$ of $(\gamma,\Phi,F)$ has the following properties:
 \begin{enumerate}
  \item $\tilde\gamma \in P(\Rc)$, that is, $\tilde \gamma$ is an isomorphism between finite induced subgraphs;
  \item $F \subset d(\tilde{\gamma})\cap r(\tilde{\gamma})$, and for all $\varphi\in\Phi$, $d(\varphi)\cup r(\varphi)\subset d(\tilde{\gamma})\cap r(\tilde{\gamma})$;
  \item The sets $\tilde\gamma^{-1}\big(d(\gamma)\setminus r(\gamma)\big)$, $\tilde\gamma\big(r(\gamma)\setminus d(\gamma)\big)$ and $K$ are pairwise disjoint and $u\nsim v$, $\forall u \in \tilde\gamma^{-1}\big(d(\gamma)\setminus r(\gamma)\big)$, $\forall v\in \tilde\gamma\big(r(\gamma)\setminus d(\gamma)\big)$;
  \item Let $\widetilde{\Phi}:=\left(\Phi\setminus\{\gamma\}\right)\cup\{\widetilde{\gamma}\}$. For any subset $A$ of $\bigcup_{\varphi\in\widetilde{\Phi}}(d(\varphi)\cup r(\varphi))$ let $\Omega$ (resp. $\widetilde{\Omega}$) be the orbit of $A$ under the groupoid generated by $\Phi$ (resp. $\widetilde{\Phi}$). Then, any edge of $\Rc$ contained in $\widetilde{\Omega}\setminus \Omega$ is the image under $\widetilde{\gamma}$ or $\widetilde{\gamma}^{-1}$ of an edge in $\Omega$.
 \end{enumerate}
\end{proposition}

\begin{proof}
\noindent(1) It is clear from the construction that $K$ is finite, as it is a finite union of finite sets. Hence, $d(\tilde\gamma) = K \sqcup \{u_1,\ldots,u_m\}$ and $r(\tilde\gamma) = \{y_1,\ldots,y_m\} \sqcup K$ are finite. Let us now check that $\tilde \gamma$ is an isomorphism between finite induced subgraphs. Notice that one has $d(\tilde\gamma)=\{x_1,\ldots,x_m\}\sqcup d(\gamma) \sqcup \{u_1,\ldots,u_m\}$ and let $u,x\in d(\tilde\gamma)$. We are going to check that $u\sim x \Leftrightarrow \tilde\gamma(u)\sim\tilde\gamma(x)$ by distinguishing cases. If $u$ and $x$ are in the same component of the disjoint union $\{x_1,\ldots,x_m\} \sqcup d(\gamma)\sqcup \{u_1,\ldots,u_m\}$, the equivalence follows readily from the construction and the fact that $\gamma\in P(\Rc)$. If $x\in d(\gamma)$ and $u=u_j$, then $x\sim u_j \Leftrightarrow \gamma(x) \sim v_j \Leftrightarrow \tilde \gamma(x) \sim \tilde \gamma(u_j)$ by the selection of the $u_j$'s. If $x=x_i$ and $u\in d(\gamma)$, then $x_i\sim u \Leftrightarrow y_i \sim\gamma(u) \Leftrightarrow \tilde \gamma(x_i) \sim \tilde \gamma(u)$ again by the choice of the $y_i$'s. Finally, if $x=x_i$ and $u=u_j$, then $\tilde\gamma(u_j) = v_j$  and $\tilde\gamma(x_i) = y_i$. It follows from the selection of the $u_j$'s and $y_i$'s that $u_j\not\sim x_i$ and $v_j\not\sim y_i$. Hence the equivalence holds.
\vspace{0.2cm}

\noindent(2) It is clear since $F\cup d(\varphi)\cup r(\varphi)$ for $\varphi\in\Phi$ are contained in $K$ and $K\subset d(\widetilde{\gamma})\cap r(\widetilde{\gamma})$.
\vspace{0.2cm}

\noindent(3) Notice that $\tilde\gamma^{-1}\big(d(\gamma)\setminus r(\gamma)\big) \subset \{u_1,\ldots,u_m\}$ and $\tilde\gamma\big(r(\gamma)\setminus d(\gamma)\big) \subset \{y_1,\ldots,y_m\}$. 
Therefore the assertion follows from the construction.
\vspace{0.2cm}

\noindent(4) Note that $\widetilde{\Omega}\setminus\Omega\subset\widetilde{\gamma}^{-1}(\Omega\setminus r(\gamma))\cup\widetilde{\gamma}(\Omega\setminus d(\gamma))$.  Let $(x,y)\in\widetilde{\Omega}^2\setminus \Omega^2$ such that $x\sim y$. If $x,y\in\widetilde{\gamma}(\Omega\setminus d(\gamma))$ or $x,y\in\widetilde{\gamma}^{-1}(\Omega\setminus r(\gamma))$ the result is trivial. Note that we cannot have $x\in\widetilde{\gamma}(\Omega\setminus d(\gamma))$ and $y\in\widetilde{\gamma}^{-1}(\Omega\setminus r(\gamma))$ because otherwise there exists $i$ and $j$ such that $x=y_i$ and $y=u_j$ which implies that $x\nsim y$. Suppose now that $x\in\widetilde{\gamma}(\Omega\setminus d(\gamma))$ and $y\in\Omega$. Then there is some $i$ such that $x=y_i$ but then $y_i\sim y$ implies that $y\in r(\gamma)$ and $\widetilde{\gamma}^{-1}(y)\sim\widetilde{\gamma}^{-1}(x)$. Since $\widetilde{\gamma}^{-1}(y)=\gamma^{-1}(y)\in\Omega$ and $\widetilde{\gamma}^{-1}(x)=\widetilde{\gamma}^{-1}(y_i)=x_i\in\Omega$, we are done. The other cases are proved in the same way.
\end{proof}

\vspace{0.2cm}

\subsection{``Treezation'' of a free group action}\label{treezation}

\noindent We denote by $\F_k$ the free group on $k$ generators $a_1,\ldots,a_k$. Given a tuple $\bar\alpha=(\alpha_1,\dots,\alpha_k)\in\autR^k$, we denote by $\alpha:\F_k\to\autR$ the unique group homomorphism such that $\alpha(a_j)=\alpha_j$ for all $j$.

\vspace{0.2cm}

\noindent In this section, given such a $k$-tuple $\bar\alpha\in\autR^k$ and a finite set $F$ of vertices of $\Rc$, we explain how to get a $k$-tuple $\bar\beta\in\autR^k$ such that, for all $j$ and $\varepsilon=\pm 1$, the automorphisms $\alpha_j^\varepsilon$ and $\beta_j^\varepsilon$ coincide on $F$, and, informally speaking, the associated Schreier graph (see definition below) looks like a tree far from $F$. 

\vspace{0.2cm}

\noindent First, we define, for $1\leq j\leq k$, the element $\beta_{0,j}\in  P(\Rc)$ by the restriction:
$$\beta_{0,j}:=\alpha_j\vert_{\alpha_j^{-1}(F)\cup F}\,:\,\alpha_j^{-1}(F)\cup F\rightarrow F\cup \alpha_j(F).$$ Write $\widetilde{F}=\bigcup_{j=1}^k\left(\alpha_j^{-1}(F)\cup F\cup\alpha_j(F)\right)=\bigcup_{j=1}^k(d(\beta_{0,j})\cup r(\beta_{0,j}))$. Let $\mathcal{B}_0\subset P(\Rc)$ be the groupoid generated by $\beta_{0,j}$ for $1\leq j\leq k$ and note that, for all $\varphi\in\mathcal{B}_0$, one has $d(\varphi)\cup r(\varphi)\subset \widetilde{F}$.

\vspace{0.2cm}

\begin{remark}\label{RmkOrbits}
Suppose that $ \alpha$ has all orbits infinite. It implies that for all $x\in \widetilde{F}$ there exists $g\in\mathcal{B}_0$ and $1\leq j\leq k$ such that $x\in d(g)$ and $gx\notin d(\beta_{0,j})\cap r(\beta_{0,j})$. Indeed, since the $\alpha$-orbit of $x$ is infinite, there exists $w\in\F_k$ such that $\alpha(w)x\notin\widetilde{F}$. Write $w=\alpha_{i_n}^{\epsilon_n}\dots\alpha_{i_1}^{\epsilon_1}$ its reduced expression, where $\epsilon_l\in\{-1,1\}$. Since $\alpha(w)x\notin\widetilde{F}$ it follows that $x\notin d(\beta_{0,i_n}^{\epsilon_n}\dots\beta_{0,i_1}^{\epsilon_1})$. We may and will assume that $x\in d(\beta_{0,i_1}^{\epsilon_1})$ (otherwise the conclusion is obvious). Let $1\leq n_0\leq n-1$ be the largest integer such that $x\in d(\beta_{0,i_{n_0}}^{\epsilon_{n_0}}\dots\beta_{0,i_1}^{\epsilon_1})$. Defining $g=\beta_{0,i_{n_0}}^{\epsilon_{n_0}}\dots\beta_{0,i_1}^{\epsilon_1}\in\mathcal{B}_0$ and $j=i_{n_0+1}$ we reach the conclusion.

\end{remark}

\noindent Let $V(\Rc)=\{z_0,z_1,\dots\}$ be an enumeration of $V(\Rc)$ and define inductively $\bar\beta_l\in P(\Rc)^k$ by $\bar\beta_0=(\beta_{0,1},\dots,\beta_{0,k})$ and, if for some $l\geq 0$ the element $\bar\beta_l=(\beta_{l,1},\dots,\beta_{l,k})$ is defined, we denote by $j(l)$ the unique element in $\{1,\dots k\}$ such that $j(l)\equiv l+1$ (mod $k$) and we define $\bar\beta_{l+1}=(\beta_{l+1,1},\dots,\beta_{l+1,k})$ where $\beta_{l+1,j(l)}$ is the elementary extension of the triple $(\beta_{l,j(l)},\{\beta_{l,1},\dots,\beta_{l,k}\},\{z_l\})$ and $\beta_{l+1,j}=\beta_{l,j}$ for $j\neq j(l)$. We will denote by $\mathcal{B}_l\subset P(\Rc)$ the groupoid generated by $\beta_{l,1},\dots,\beta_{l,k}$.

\vspace{0.2cm}

\noindent Define $\bar \beta=(\beta_1,\dots,\beta_k)\in\text{Aut}(\Rc)^k$, where $\beta_j$ is (well-) defined by $\beta_j(x)=\beta_{j,l}(x)$ whenever $x\in d(\beta_{j,l})$ (and denote by $\beta\,:\,\F_k\rightarrow\text{Aut}(\Rc)$ the unique group morphism which maps the $a_j$ onto $\beta_j$).

\vspace{0.2cm}

\begin{definition}
 The $k$-tuple $\bar\beta$ is said to be a \emph{treezation} of $\bar\alpha$ relatively to $F$.
\end{definition}

\noindent We denote by $\mathcal{G}_\beta$ the \textit{Schreier graph of} $\beta$: the vertices are $V(\mathcal{G}_\beta)=V(\Rc)$ and for all $x\in V(\mathcal{G}_\beta)$ we have an oriented edge from $x$ to $\beta_j(x)$ which is decorated by $j^+$ and an oriented edge from $\beta_j^{-1}(x)$ to $x$ decorated by $j^-$.

\vspace{0.2cm}

\noindent More generally, given a groupoid $H\subset P(\Rc)$ generated by $\gamma_1,\dots,\gamma_k\in P(\Rc)$, we define $\mathcal{G}_H$ the \textit{Schreier graph of} $H$ with the generating tuple $(\gamma_1,...,\gamma_k)$ in the following way: $V(\mathcal{G}_H)=V(\Rc)$ and, for all $x\in V(\mathcal{G}_H)$, there is an oriented edge decorated by $j^+$ from $x$ to $\gamma_j(x)$ whenever $x\in d(\gamma_j)$ and there is an oriented edge decorated by $j^-$ from $\gamma_j^{-1}(x)$ to $x$ whenever $x\in r(\gamma_j)$. For all $l\geq 0$, we will denote by $\mathcal{G}_l$ the Schreier graph of $\mathcal{B}_l$ with the generating tuple $\bar\beta_l$.

\vspace{0.2cm}

\noindent Given a graph $\mathcal{G}$, for $l\geq 2$, a \textit{minimal path} in $\mathcal{G}$ from $x_1\in V(\mathcal{G})$ to $x_l\in V(\mathcal{G})$ is a finite sequence of pairwise distinct vertices $x_1,\dots x_{l}$ such that $x_i\sim x_{i+1}$ for all $1\leq i\leq l-1$. When $l\geq 3$ and $x_1\sim x_l$, we call it a \textit{minimal cycle}.

\vspace{0.2cm}

\noindent Recall that $a_1,\dots a_k$ are the canonical generators of $\F_k$.
\begin{proposition}\label{treezationProperties}
Any treezation $\bar\beta$ of $\bar\alpha$ satisfies the following properties.
\begin{enumerate}
\item For all $l\geq 0$ the minimal cycles in $\mathcal{G}_l$ are all contained in $\widetilde{F}$.
\item The minimal cycles of $\mathcal{G}_\beta$ are all in $\widetilde{F}$.
\item The minimal paths of $\mathcal{G}_\beta$ with extremities in $\widetilde{F}$ are contained in $\widetilde{F}$.
\item If the orbits of $\alpha$ are infinite then the orbits of $\beta$ are infinite.
\item For all $x,y\in\beta(\F_k)(\widetilde{F})$, if $(x,y)\in E(\Rc)$ then there exists $x_0,y_0\in\widetilde{F}$ such that $(x_0,y_0)\in E(\Rc)$ and $w\in\F_k$ such that $\beta(w)x_0=x$ and $\beta(w)(y_0)=y$.
\item Let $x\in\widetilde{F}$ and $w=a_{i_n}^{\epsilon_n}\dots a_{i_1}^{\epsilon_1}\in\F_k$ with its reduced expression, where $\epsilon_l\in\{-1,1\}$. If there exists $s$ such that $\beta_{i_s}^{\epsilon_s}\dots \beta_{i_1}^{\epsilon_1}x\notin\widetilde{F}$, then for all $s<t\leq n$ the path
$$\beta_{i_s}^{\epsilon_s}\dots \beta_{i_1}^{\epsilon_1}x,\, \beta_{i_{s+1}}^{\epsilon_{s+1}}\dots \beta_{i_1}^{\epsilon_1}x,\,\dots,\,\beta_{i_n}^{\epsilon_n}\dots \beta_{i_1}^{\epsilon_1}x$$
is a geodesic path in $\mathcal{G}_\beta$ and $d(\beta_{i_t}^{\epsilon_t}\dots \beta_{i_1}^{\epsilon_1}x,\widetilde{F})=1+d(\beta_{i_{t-1}}^{\epsilon_{t-1}}\dots \beta_{i_1}^{\epsilon_1}x,\widetilde{F})$ for all $s<t\leq n$, where $d$ is the graph distance.
\end{enumerate}
\end{proposition}

\begin{proof}
\noindent(1) The result is obvious for $\mathcal{G}_0$ since all the edges in $\mathcal{G}_0$ have their source and range in $\widetilde{F}$. Now, observe that the edges of $\mathcal{G}_l$ are included in the edges of $\mathcal{G}_{l+1}$ with the same label and any edge of $\mathcal{G}_{l+1}$ which is not already an edge on $\mathcal{G}_l$ has an extremity which is not in any domain or range of any $\beta_{l,j}$ for $1\leq j\leq k$ since $\beta_{l+1,j(l)}$ is an elementary extension of $\beta_{l,j(l)}$ and for $j\neq j(l)$ one has $\beta_{l+1,j}=\beta_{j,l}$. Hence, the minimal cycles in $\mathcal{G}_{l+1}$ are contained in $\mathcal{G}_l$. This proves the result by induction on $l$.
\vspace{0.2cm}

\noindent(2) It follows from (1) since any minimal cycle in $\mathcal{G}_\beta$ is a minimal cycle in $\mathcal{G}_l$ for some $l\geq 0$.
\vspace{0.2cm}

\noindent(3) The proof is the same as the one of  (1) and (2).
\vspace{0.2cm}

\noindent(4) It follows from (2) that the induced graph structure on the complement of $\widetilde{F}$ coming from $\mathcal{G}_\beta$ is a forest. Hence, every $x\notin\widetilde{F}$ has an infinite $\beta$-orbit. If $x\in\widetilde{F}$, we find, by Remark \ref{RmkOrbits}, $g\in\mathcal{B}_0$ and $1\leq j\leq n$ such that $x\in d(g)$ and $gx\notin d(\beta_{0,j})\cap r(\beta_{0,j})$. Hence either $\beta_{k,j}(gx)\notin\widetilde{F}$ or $\beta_{k,j}^{-1}(gx)\notin\widetilde{F}$. In both cases, we find an element $w\in\F_k$ such that $\beta(w)(x)\notin\widetilde{F}$. By the first part of the proof, the $\beta$-orbit of $x$ is infinite.
\vspace{0.2cm}

\noindent(5) Let $\Omega$ be the orbit of $\widetilde{F}$ under the action $\beta$ and, for $l\geq 0$, $\Omega_l$ be the orbit of $\widetilde{F}$ under the groupoid $\mathcal{B}_l$. Since every edge of $\Rc$ which is in $\Omega$ is actually in $\Omega_l$ for some $l$, it suffices to show that, for all $l\geq 0$ and all $x,y\in\Omega_l$ with $(x,y)\in E(\Rc)$, there exists $g\in\mathcal{B}_l$ and $x_0,y_0\in\widetilde{F}$ such that $gx_0=x$ and $gy_0=y$. For $l=0$ it is trivial since $\Omega_0=\widetilde{F}$ and the proof follows by induction by using assertion (4) of Proposition \ref{PropElExt}.
\vspace{0.2cm}

\noindent(6) The proof is a direct consequence of the following remark, which is itself a direct consequence of (2) and (3). For all $x\notin\widetilde{F}$ in the connected component of a point in $\widetilde{F}$ there exist a unique path in $\mathcal{G}_\beta$ without backtracking $x_1,x_2,\dots,x_l$ such that
\begin{itemize}
\item $x_1=x$,
\item $x_2,\dots,x_{l-1}\notin\widetilde{F}$,
\item $x_l\in\widetilde{F}$.
\end{itemize}
Moreover, this path is geodesic.
\end{proof}
\subsection{Proof of Theorem D}

\noindent Let us fix an integer $k\geq 2$. We still denote by $\F_k$ the free group on $k$ generators $a_1,\ldots,a_k$, and by $\alpha$ the morphism $\F_k\rightarrow\autR$ associated to some $\bar\alpha\in\autR^k$. Let us mention that one has $\langle \bar\alpha\rangle = \{\alpha(w): \, w\in\F_ k\}$. The set 
\[
 \Ac=\{\bar\alpha\in\aut(\Rc)^k: \text{ every } \langle \alpha\rangle \text{-orbit on the vertices is infinite}\} \,  
\]
is closed in $\aut(\Rc)^k$, which is closed on S$(\Rc)^k$. Thus $\Ac$ is a Baire space. We also consider the subsets
\begin{eqnarray*}
 \Fc & = & \{\bar\alpha\in \aut(\Rc)^k: \, \alpha \text{ is injective } \} \text; \\
 \Hc & = & \{\bar\alpha\in \aut(\Rc)^k: \, \alpha \text{ is an homogeneous action on } \aut(\Rc) \}  \text.
\end{eqnarray*}
Now, Theorem D can be stated as follows.

\begin{theorem}\label{free-acts-homogeneously}
The subset $\Ac\cap \Fc \cap \Hc$ is a dense $G_{\delta}$ in $\Ac$. 
\end{theorem}

\begin{proof}
Let us prove that $\Ac\cap \Hc$ is a dense $G_{\delta}$ in $\Ac$. To do this, we observe that $\Ac\cap\Hc=\bigcap_{\varphi\in P(\Rc)} (\Ac\cap \Hc_\varphi)$, where 
$
 \Hc_\varphi = \{\bar\alpha\in\aut(\Rc)^k: \, \exists w\in\F_k, \, \forall x\in d(\varphi), \ \alpha(w) x= \varphi x\}.
$
Since the $\Ac \cap \Hc_\varphi$'s are all open because the $\Hc_\varphi$'s are open in $\aut(\Rc)^k$, it suffices to prove that $\Ac \cap \Hc_\varphi$ is dense for all $\varphi\in P(\Rc)$.
 
\vspace{0.2cm}
 
\noindent Take an arbitrary $\varphi\in P(\Rc)$ and an arbitrary $\bar\alpha=(\alpha_1,\dots, \alpha_k)\in \Ac$. We need to prove that, for any finite subset $F\subset V(\Rc)$, there exists $\bar\omega\in\Ac\cap\Hc_\varphi$ such that  $\alpha_j^{\pm 1}$ and $\omega_j^{\pm 1}$ coincide on $F$ for all $j$. To prove this, we may assume that $F$ contains both $d(\varphi)$ and $r(\varphi)$.
 
\vspace{0.2cm}
 
\noindent We take such an $F$ and we set $\tilde F = \bigcup_{j=1}^k (\alpha_j^{-1}(F)\cup F \cup \alpha_j(F))$ as we did in Section \ref{treezation}. Then we take a treezation $\bar\beta$ of $\bar\alpha$ relatively to $F$. Let us recall that $\beta$ has all orbits infinite since $\alpha$ does, and that, for all $j$, the automorphisms $\beta_j^{\pm 1}$ and $\alpha_j^{\pm 1}$ coincide on $F$. 

\vspace{0.2cm}

\noindent Recall that (known as Neumann's Lemma \cite{Neu76}), if every orbit of an group action  $G\curvearrowright X$ in infinite, then for every finite subset $F$ of $X$, there exists $g\in G$ such that $gF\cap F=\emptyset$. Hence, there exists $u\in\F_ k$ (we consider it as a reduced word) such that $\beta(u)\tilde F \cap \tilde F = \emptyset$. Up to replacing $u$ by some $a_j u$, we may assume that $u$ is moreover cyclically reduced. This is possible thanks to Proposition \ref{treezationProperties} (6). Since $u^2$ is reduced, we also have $\beta(u^2)\tilde F \cap \tilde F = \emptyset$, again by Proposition \ref{treezationProperties} (6). 
 
\vspace{0.2cm}
 
\noindent Let us now show that $\tilde F$ and $\beta(u^2)\tilde F$ are disconnected, that is, for all $x\in \tilde F$ and $y\in \beta(u^2)\tilde F$, we have $x\nsim y$. Assume by contradiction that there exists $x\in \tilde F$ and $y\in \beta(u^2)\tilde F$ such that $x\sim y$. By Proposition \ref{treezationProperties} (5), there exist $x_0,y_0 \in \tilde F$ and $v\in \F_k$ such that $x=\beta(v)x_0$ and $y=\beta(v)y_0$. Using Proposition \ref{treezationProperties} (6), we see first that $v=uv'$ (reduced expression) since $y=\beta(v)y_0$ is in $\beta(v)\tilde F \cap \beta(u^2)\tilde F$, and then that $\beta(v)x_0 = \beta(u)\beta(v')x_0$ is not in $\tilde F$. Indeed, if $\beta(v')x_0\in\tilde F$, then $\beta(u)\beta(v')x_0$ is not in $\tilde F$, and if $\beta(v')x_0\notin\tilde F$, then $\beta(u)\beta(v')x_0$ is farer from $\tilde F$ than $\beta(v')x_0$ in the Schreier graph of $\beta$. Since we have $\beta(v)x_0\notin \tilde F$, $x\in \tilde F$ and $x=\beta(v)x_0$, we have obtained a contradiction.

\vspace{0.2cm}

\noindent We now set $K = \bigcup_v \beta(v) \tilde F$, where the union is taken over the prefixes of $u^2$ (it is the union of trajectories (see Definition 3.1 in \cite{MS13}) of points in $\tilde F$ in the Schreier graph of $\beta$) and $\tilde K = \bigcup_{j=1}^k (\beta_j^{-1}(K)\cup K \cup \beta_j(K))$. We also consider some integer $s>10\operatorname{diam}(\tilde K)$, with respect to the metric on the Schreier graph of $\beta$, and an index $i$ such that the first letter of $u$ is different from $a_i^{\pm 1}$.  Proposition \ref{treezationProperties} ensures that $\beta(a_i^{-s}u^2)\tilde F$, $\tilde K$ and $\beta(a_i^{s}u^2)\tilde F$ are pairwise disjoint. 

\vspace{0.2cm}

\noindent Moreover, $\beta(a_i^{-s}u^2)\tilde F$ and $\beta(a_i^{s}u^2)\tilde F$ are both disconnected from $\tilde K$. Indeed, assume by contradiction that there exist $x\in \tilde K$ and $y\in \beta(a_i^su^2)\tilde F$ such that $x\sim y$. By Proposition \ref{treezationProperties} (5), there exist $x_0,y_0 \in \tilde F$ and $v\in \F_k$ such that $x=\beta(v)x_0$ and $y=\beta(v)y_0$. On the one hand, using Proposition \ref{treezationProperties} (6), we see first that $v=a_i^suv'$ (reduced expression) since $y=\beta(v)y_0$ is in $\beta(v)\tilde F \cap \beta(a_i^su^2)\tilde F$, and then that the distance (in the Schreier graph of $\beta$) between $\beta(v)x_0$ and $\tilde F$ is greater than $10\operatorname{diam}(\tilde K)$. Indeed, we get, as above, that $\beta(uv')x_0\notin \tilde F$, and deduce that the distance between $\beta(v)x_0 = \beta(a_i^s)\beta(uv')x_0$ and $\tilde F$ is greater than $s$, which is itself greater than $10\operatorname{diam}(\tilde K)$. On the other hand, since $x$ is in $\tilde K$, its distance is less or equal to $\operatorname{diam}(\tilde K)$. Hence, we get $x\neq \beta(v)x_0$, a contradiction. This proves that there are no edges between $\tilde K$ and $\beta(a_i^{s}u^2)\tilde F$. One gets similarly that there are no edges between $\tilde K$ and $\beta(a_i^{-s}u^2)\tilde F$.

\vspace{0.2cm}

\noindent Thus we can define a partial graph isomorphism $\tau\in P(\Rc)$ as follows: we fix some index $i'$ distinct from $i$ (which exists since $k\geq 2$) and set:
\begin{itemize}
 \item $d(\tau) = (\beta_{i'}^{-1} K\cup K) \sqcup \beta(a_i^{s}u^2)d(\varphi)$; \quad $\quad r(\tau) = (K\cup\beta_{i'}K) \sqcup \beta(a_i^{-s}u^2)r(\varphi)$;
 \item for $x\in \beta_{i'}^{-1} K\cup K$, we set $\tau x=\beta_{i'}x$;
 \item for $x\in \beta(a_i^{s}u^2)d(\varphi)$, we set $\tau x= \beta(a_i^{-s}u^2) \circ \varphi \circ \beta(u^{-2} a_i^{-s})x$.
\end{itemize}
Note that without the fact that $\beta(a_i^{-s}u^2)\tilde F$ and $\beta(a_i^{s}u^2)\tilde F$ are both disconnected from $\tilde K$, the map $\tau$ could be only a bijection between vertices.

\vspace{0.2cm}

\noindent Now, one can extend $\tau$ to an automorphism $\omega_{i'}$, and set $\omega_j=\beta_j$ for other indices $j$. Then, for all $j$, $\omega_j^{\pm 1}$, $\beta_j^{\pm 1}$ and $\alpha_j^{\pm 1}$ coincide on $F$. Moreover $\beta(a_i^{-s}u^2)$ and $ \omega(a_i^{-s}u^2)$ coincide on $F$, and the same holds for $\beta(a_i^{s}u^2)$ and $\omega(a_i^{s}u^2)$. Hence, one has
\begin{eqnarray*}
 \omega(u^{-2}a_i^{-s}a_{i'}a_i^{-s}u^2)x &=& \beta(u^{-2}a_i^{-s}) \tau \beta(a_i^{-s}u^2)x \\
 &=& \beta(u^{-2}a_i^{-s})\beta(a_i^{s}u^2) \varphi \beta(u^{-2} a_i^s)\beta(a_i^{-s}u^2)x= \varphi x
\end{eqnarray*}
for all $x\in d(\varphi)$. In other words, $\bar \omega=(\omega_1,\ldots,\omega_k)$ is in $\Hc_\varphi$.
 
\vspace{0.2cm}
 
\noindent Finally, as $\omega_{i'}$ coincides with $\beta_{i'}$ on $\beta_{i'}^{-1}K \cup K$ and all $\beta_i$-orbits are infinite outside this set, $\omega$ also has all orbits infinite, that is $\bar\omega$ is in $\Ac$. Consequently, we obtain that $\omega$ is in $\Ac\cap\Hc_\varphi$, which concludes the proof that $\Ac\cap\Hc$ is dense $G_\delta$ in $\Ac$. One can use a similar, but much easier, argument to show that $\Ac\cap \Fc$ is a dense $G_{\delta}$ in $\Ac$. This is left to the reader.
\end{proof}

\end{document}